\documentclass[a4paper,UKenglish,cleveref, autoref, thm-restate,authorcolumns]{lipics-v2019}
\nolinenumbers

\usepackage{tikz}
\usepackage{tikz-cd}
\usepackage[usestackEOL]{stackengine}
\usepackage{hyperref}
\usepackage{xcolor}

\hypersetup{
 colorlinks=true, 
 linktoc=all,  
 linkcolor=blue,  
}
\usepackage{fdsymbol}
\usepackage{graphics}

\PassOptionsToPackage{table,dvipsnames}{xcolor}
\usetikzlibrary{decorations.pathreplacing
   ,decorations.pathmorphing
   ,shapes.geometric
   ,chains
   ,matrix
   ,positioning
   , decorations.markings, intersections
   ,scopes }

\tikzset{%
    symbol/.style={%
        draw=none,
        every to/.append style={%
            edge node={node [sloped, allow upside down, auto=false]{$#1$}}}
    }
}

\definecolor{vioteal}{RGB}{90,140,220}
\tikzstyle{every picture}=[semithick, scale = 0.7, baseline = (current bounding box.center)]
\tikzset{t/.style= {draw=white, double = teal, ultra thick}} 
\tikzset{vt/.style= {draw=white, double = vioteal, ultra thick}} 
\tikzset{s/.style= {draw=white, double = red, ultra thick}} 
\tikzset{f/.style= {draw=white, double = black, ultra thick}} 
\tikzset{s0/.style= {red, semithick}} 
\tikzset{t0/.style= {teal, semithick}} 
\tikzset{f0/.style= {black, semithick}} 
\tikzset{ts/.style = {violet, ultra thick} } 
\tikzset{s-scope/.style={every path/.style=s}}	

\usepackage{tikz-cd}
\usetikzlibrary{arrows}
\usetikzlibrary{arrows.meta}
\usetikzlibrary{decorations.pathmorphing,shapes}
\usetikzlibrary{decorations.markings}
\tikzset{shorten <>/.style={shorten >=#1,shorten <=#1}}

\newcommand{\catname}[1]{{\normalfont\textbf{#1}}}

\newcommand{\compose}{ }
\newcommand{\postp}[1]{#1{\RHD}}
\newcommand{\prep}[1]{{\LHD}{#1}}
\newcommand{\Cat}{\catname{Cat}}
\newcommand{\Pushmap}[1]{\lambda_{#1}}
\newcommand{\push}[2]{\exists_{#1}{[\sectionp{#2}]}} 

\newcommand{\sectionp}[1]{\smallstar_{#1}}

\newcommand{\smallstar}{\smallwhitestar}
\newcommand{\true}{\mathrm{true}}
\newcommand{\false}{\mathrm{false}}
\newcommand{\Pred}{\mathbf{Pred}}
\newcommand{\Rel}{\mathbf{Rel}}
\newcommand{\Set}{\mathbf{Set}}
\newcommand{\injection}[2]{\iota_{#1,#2}}
\newcommand{\projection}[2]{\pi_{#1,#2}}
\newcommand{\comprehension}[2]{[#1,#2]}

\newcommand{\comprehensionzero}{[-]}
\newcommand{\injectionzero}{\iota}
\newcommand{\projectionzero}{\pi}
\newcommand{\quotient}[1]{[\hspace{-.22em}[#1]\hspace{-.22em}]}
\newcommand{\Bcategory}{\mathscr{B}}
\newcommand{\Ccategory}{\mathscr{C}}

\newcommand{\Ecategory}{\mathscr{E}}
\newcommand{\Ktwocategory}{\mathscr{K}}
\newcommand{\Barrowcategory}{\mathscr{B}^{\to}}
\newcommand{\jacobscomprehension}{\mathscr{P}}
\newcommand{\fumexquotient}{\mathscr{Q}}
\newcommand{\codfunctor}{\mathbf{cod}}
\newcommand{\domfunctor}{\mathbf{dom}}
\newcommand{\id}[1]{\mathrm{id}_{#1}}
\newcommand{\idzero}[1]{\mathrm{id}}
\newcommand{\imagefunctor}{\mathbf{image}}

\newcommand{\Alg}[1]{\mathbf{Alg}_{#1}}
\newcommand{\CoAlg}[1]{\mathbf{CoAlg}_{#1}}
\newcommand{\ArrowLax}[1]{\mathbf{#1//#1}}
\newcommand{\ArrowStrict}[1]{\mathbf{#1/#1}}

\newcommand{\Endo}[1]{\catname{Endo}(#1)}

\newcommand{\Bobject}{\mathrm{\mathbf{B}}}
\newcommand{\Eobject}{\mathrm{\mathbf{E}}}
\newcommand{\Xobject}{\mathrm{\mathbf{X}}}
\newcommand{\Bobjectarrow}{\mathrm{\mathbf{B}}^{\to}}
\newcommand{\Ffunctor}{F}
\newcommand{\Gfunctor}{G}

\renewcommand{\compose}{\circ}
\renewcommand{\postp}[1]{{#1}\rhd}
\renewcommand{\prep}[1]{\lhd{#1}}
\newcommand{\source}{\mathbf{source}}
\newcommand{\target}{\mathbf{target}}
\newcommand{\arrowdom}{\mathbf{dom}}
\newcommand{\arrowcod}{\mathbf{cod}}
\newcommand{\arrowid}{\mathbf{id}}
\newcommand{\arrowcell}{\mathbf{hom}}

\newcommand{\Triple}[3]{({#1}, {#3}, {#2})}

\title{Comprehension and quotient structures\\
in the language of 2-categories} 
\titlerunning{Comprehension and quotient structures in the language of 2-categories} 

\author{Paul-Andr{\'e} Melli{\`e}s}{CNRS, Institut de Recherche en Informatique Fondamentale (IRIF), Universit\'e de Paris, France \and \url{http://www.irif.fr/~mellies} }{mellies@irif.fr}{https://orcid.org/0000-0001-6180-2275}{ERC Advanced Project DuaLL, Grant agreement ID: 670624.}

\author{Nicolas Rolland}{Institut de Recherche en Informatique Fondamentale (IRIF), Universit\'e de Paris, France}{nicolas.rolland@gmail.com}{https://orcid.org/0000-0002-0539-3718}{}

\authorrunning{Paul-Andr{\'e} Melli{\`e}s and Nicolas Rolland} 

\Copyright{Paul-Andr{\'e} Melli{\`e}s and Nicolas Rolland} 



\relatedversion{} 

\supplement{}

\EventEditors{Zena M. Ariola}
\EventNoEds{1}
\EventLongTitle{5th International Conference on Formal Structures for Computation and Deduction (FSCD 2020)}
\EventShortTitle{FSCD 2020}
\EventAcronym{FSCD}
\EventYear{2020}
\EventDate{June 29--July 5, 2020}
\EventLocation{Paris, France}
\EventLogo{}
\SeriesVolume{167}
\ArticleNo{6}

\begin{document}

\maketitle
\begin{abstract}
Lawvere observed in his celebrated work on hyperdoctrines that the set-theoretic
schema of comprehension can be elegantly expressed in the functorial language of categorical logic,
as a comprehension structure on the functor $p:\mathscr{E}\to\mathscr{B}$
defining the hyperdoctrine.
In this paper, we formulate and study a strictly ordered hierarchy of three notions
of comprehension structure on a given functor $p:\mathscr{E}\to\mathscr{B}$,
which we call (i)~comprehension structure, (ii)~comprehension structure with section,
and (iii)~comprehension structure with image.
Our approach is 2-categorical and we thus formulate
the three levels of comprehension structure 
on a general morphism $p:\mathrm{\mathbf{E}}\to\mathrm{\mathbf{B}}$ in a 2-category~$\mathscr{K}$.
This conceptual point of view on comprehension structures enables us
to revisit the work by Fumex, Ghani and Johann on the duality
between comprehension structures and quotient structures on a given functor $p:\mathscr{E}\to\mathscr{B}$.
In particular, we show how to lift the comprehension and quotient structures 
on a functor $p:\mathscr{E}\to\mathscr{B}$ to the categories of algebras or coalgebras
associated to functors $F_{\mathscr{E}}:\mathscr{E}\to\mathscr{E}$
and $F_{\mathscr{B}}:\mathscr{B}\to\mathscr{B}$ of interest, 
in order to interpret reasoning by induction and coinduction in the traditional language of categorical logic, 
formulated in an appropriate 2-categorical way.
\end{abstract}

\section{Introduction}

\paragraph*{A fundamental duality between comprehension and quotient structures}
One fundamental discovery by Lawvere \cite{lawvere-hyperdoctrines} is that the \emph{comprehension schema} 
of Zermelo set theory \cite{zermelo} can be elegantly expressed in the functorial language of categorical logic, in the following way.
Consider
the category $\Set$ of sets and functions, and
the category $\Pred$ of predicates,
defined in the following way: its objects are the pairs $(A,R)$ consisting of a set~$A$
and of a function $R:A\to\Omega$ to the set $\Omega=\{\false,\true\}$ of booleans, 
describing a specific predicate~$R$ of~$A$ ; 
its morphisms $f:(A,R)\to(B,S)$ are the functions $f:A\to B$
such that $\forall a\in A, Ra\Rightarrow S(fa)$.
%
The functor $p:\Pred\to\Set$ is the forgetful functor which transports every predicate $(A,R)$ to its underlying set~$A$.
The comprehension schema enables one to turn every predicate $(A,R)$ 
into a set $\comprehension{A}{R}$ defined as follows
$$
\comprehension{A}{R} \,\, := \,\, \{ \,\, a\in A \,\, | \,\, Ra=\true \,\, \}
$$
equipped moreover with a function 
$$
\begin{tikzcd}[column sep =1em]
{\injection{A}{R}} \quad : \quad {\comprehension{A}{R}} \arrow[rr] && A
\end{tikzcd}
$$
which transports every element~$a$ of the set $\comprehension{A}{R}$ to itself, seen as element in $A$.
The construction is natural in $(A,R)$ in the sense that it defines a functor
$$
\begin{tikzcd}[column sep =1em]
{\comprehensionzero} \quad : \quad \Pred \arrow[rr] && \Set
\end{tikzcd}
$$
together with a natural transformation 
$$
\begin{tikzcd}[column sep =1em]
\iota \,\, : \,\, {\comprehensionzero} \arrow[rr,-implies,double] && p \,\, : \,\, \Pred \arrow[rr] && \Set.
\end{tikzcd}
$$
Here, naturality means that every morphism $f:(A,R)\to (B,S)$
between predicates induces the commutative diagram below, in the category~$\Set$.
\begin{equation}\label{equation/square-diagram-for-comprehension}
\begin{tikzcd}[column sep =3em,row sep =1em]
{\comprehension{A}{R}} \arrow[rr,"{{\injection{A}{R}}}"]
\arrow[dd,"{{\comprehension{A}{f}}}"{swap}]
&& A\arrow[dd,"{f}"]
\\
\\
{\comprehension{B}{S}} \arrow[rr,"{{\injection{B}{S}}}"] && B
\end{tikzcd}
\end{equation}
More generally,
considering this example
of the functor $p:\Pred\to\Set$ as typical, it makes sense to formulate the following ``minimalist'' notion of comprehension structure:
\begin{definition}\label{definition/comprehension-structure}
A comprehension structure on a functor $p:\Ecategory\to\Bcategory$
is a pair $(\comprehensionzero,\injectionzero)$ consisting of a functor $\comprehensionzero:\Ecategory\to\Bcategory$
and of a natural transformation $\injectionzero:\comprehensionzero\Rightarrow p$.
\end{definition}
Interestingly, Jacobs provides in \cite{jacobs-paper} a useful and detailed survey
of a hierarchy of axiomatic requirements on a functor $p:\Ecategory\to\Bcategory$
appearing in the literature, from which such a comprehension structure can be derived.
In a decreasing order of generality, one finds:
\begin{itemize}
\item Jacob's comprehension categories, defined in \cite{jacobs-paper}, Def. 4.1, page 181.
\item Ehrhard's $D$-categories \cite{ehrhard} called comprehension categories with unit in \cite{jacobs-paper}, Def. 4.12.
\item Lawvere categories \cite{lawvere-hyperdoctrines} as Jacobs defined them in \cite{jacobs-paper}, first paragraph of p. 190.
\end{itemize}
Our definition just given of a comprehension structure (Def.~\ref{definition/comprehension-structure})
does not appear as such in the literature, at least in the elementary 2-categorical way we express it here.
The reason is that the comprehension pair $(\comprehensionzero,\injectionzero)$
can be equivalently formulated as a functor $\jacobscomprehension:\Ecategory\to\Barrowcategory$
to the category of arrows of $\Bcategory$, making the diagram below commute:
\begin{equation}\label{equation/Barrowcategory-cod-functor}
\begin{tikzcd}[column sep = .8em, row sep = .8em]
\Ecategory\arrow[rrrr,"{\jacobscomprehension}"]\arrow[rrdd,"p"{swap}] &&&& \Barrowcategory\arrow[lldd,"{\codfunctor}"]
\\
\\
&& \Bcategory
\end{tikzcd}
\end{equation}
where $\codfunctor:\Barrowcategory\to\Bcategory$ denotes the codomain functor.
The definitions of comprehension category \cite{jacobs-paper}  and of Lawvere category \cite{lawvere-hyperdoctrines} are based on this formulation, while the definition of $D$-categories  \cite{ehrhard} works in an entirely different way,
which we analyze later in this introduction, as well as in \S\ref{section/comprehension-with-section}.
One purpose of the present paper is to revisit these three levels definitions
from a purely 2-categorical point of view.
This search for a clean 2-categorical account of comprehension
in categorical logic is motivated by our desire to understand 
at this level of abstraction a recent observation 
by Fumex, Ghani and Johann~\cite{fumex-phdthesis,ghani-johann-fumex},
who establish a very nice duality between (a) the \emph{operation of comprehension}
which underlies reasoning by induction using initial algebras,
and (b) the \emph{operation of quotienting} 
which underlies reasoning by coinduction using terminal coalgebras.
In particular, Fumex introduces in his PhD thesis a notion of $tC$-opfibration $p:\Ecategory\to\Bcategory$ (where $tC$ refers to the section functor $t$ 
and the comprehension functor $C$ of the structure)
adapted for induction reasoning, which he then dualizes into a notion of $QCE$-category $p^{op}:\Ecategory^{op}\to\Bcategory^{op}$ (where $QCE$ stands for quotient category with equality)
adapted for coinduction reasoning, and simply obtained by reversing the orientation of every morphism in $\Ecategory$ and $\Bcategory$.
\medbreak

In order to understand and to illustrate this idea of quotient structures,
consider the category $\Rel$ whose objects are pairs $(A,R)$ consisting of a set $A$ and of a binary relation $R\subseteq A\times A$,
and whose morphisms $f:(A,R)\to (B,S)$ are functions $f:A\to B$ such that $R(a,a')\Rightarrow S(fa,fa')$.
As in the previous case, the functor $p:\Rel\to\Set$ is the forgetful function
which transports every binary relation $(A,R)$ to its underlying set~$A$.
Every binary relation $(A,R)$ induces a set $A/R$ defined as the quotient
$$
A/R := A/{{\sim_R}}
$$
of the underlying set $A$ by the equivalence relation $\sim_R$ generated by the binary relation~$R$.
The set $A/R$  comes together with a function
$$
\begin{tikzcd}[column sep =1em]
{\projection{A}{R}} \quad : \quad A \arrow[rr] && {A/R}
\end{tikzcd}
$$
which transports every element of $A$ to its equivalence class modulo $\sim_{R}$.
The construction is natural in $(A,R)$ in the sense that it defines a functor
$$
\begin{tikzcd}[column sep =1em]
{\quotient{-}} \quad : \quad \Rel \arrow[rr] && \Set
\end{tikzcd}
$$
with $\quotient{A,R}=A/R$, together with a natural transformation 
$$
\begin{tikzcd}[column sep =1em]
\projectionzero \,\, : \,\, p \arrow[rr,-implies,double] && {\quotient{-}} \,\, : \,\, \Rel \arrow[rr] && \Set.
\end{tikzcd}
$$
Here, naturality means that every predicate morphism $f:(A,R)\to (B,S)$
induces the commutative diagram below in the category~$\Set$.
\begin{equation}\label{equation/square-diagram-for-quotient}
\begin{tikzcd}[column sep =3em,row sep =1em]
A\arrow[rr,"{{\projection{A}{R}}}"]
\arrow[dd,"{f}"{swap}]
&& 
{\quotient{A,R}} 
\arrow[dd,"{\quotient{f}}"]
\\
\\
B \arrow[rr,"{{\projection{B}{S}}}"] && {\quotient{B,S}} 
\end{tikzcd}
\end{equation}
In the same way as previously,
this example leads us to the following definition,
obtained by dualizing Def.~\ref{definition/comprehension-structure}.
\begin{definition}\label{definition/basic-quotient}
A quotient structure on a functor $p:\Ecategory\to\Bcategory$
is a pair $(\quotient{-},\projectionzero)$ consisting of a functor $\quotient{-}:\Ecategory\to\Bcategory$
and of a natural transformation $\projectionzero:p\Rightarrow\quotient{-}$.
\end{definition}
In the same way as previously, and by duality,
a quotient structure is the same thing as a functor $\fumexquotient:\Ecategory\to\Barrowcategory$
to the category of arrows of~$\Bcategory$, making the diagram below commute:
\begin{equation}\label{equation/Barrowcategory-dom-functor}
\begin{tikzcd}[column sep = .8em, row sep = .8em]
\Ecategory\arrow[rrrr,"{\fumexquotient}"]\arrow[rrdd,"p"{swap}] &&&& \Barrowcategory\arrow[lldd,"{\domfunctor}"]
\\
\\
&& \Bcategory
\end{tikzcd}
\end{equation}
where $\domfunctor:\Barrowcategory\to\Bcategory$ denotes the domain functor.

\paragraph*{A 2-categorical classification of comprehension structures} 
\noindent
\textbf{(i) Comprehension structures.}
In order to understand the duality between comprehension and quotient structures,
we find enlightening to take seriously the 2-categorical nature of
Def.~\ref{definition/comprehension-structure}
and~\ref{definition/basic-quotient}, and to reformulate them in the following way.
Suppose given a 2-category~$\Ktwocategory$ such as $\Ktwocategory=\Cat$, the 2-category of categories.
We consider the 2-category $\ArrowLax{\Ktwocategory}$
whose objects are the triples $(\Eobject,\Bobject,p:\Eobject\to \Bobject)$
consisting of a pair of $0$-cells $\Eobject$ and $\Bobject$ and a $1$-cell $p:\Eobject\to\Bobject$
of the 2-category $\Ktwocategory$, and whose morphisms
\begin{equation}\label{equation/lax-morphisms}
\begin{tikzcd}[column sep=1em]
(f_\Eobject,f_\Bobject, \varphi) \,\,\, : \,\,\, (\Eobject_1,\Bobject_1,p_1:\Eobject_1\to \Bobject_1) \arrow[rr] && (\Eobject_2,\Bobject_2,p_2:\Eobject_2\to \Bobject_2)
\end{tikzcd}
\end{equation}
are triples consisting of a pair of $1$-cells $f_\Bobject:\Bobject_1\to \Bobject_2$ and $f_\Eobject:\Eobject_1\to \Eobject_2$
and a 2-cell
$$
\begin{tikzcd}[column sep=3em,row sep=.4em]
{\Eobject_1}  \ar[dd, "{p_1}"{swap},"{}"{name=target}] \ar[rr,"{f_\Eobject}"] 
&& 
{\Eobject_2} \ar[dd,"{p_2}","{}"{swap,name=source}] 
\\
{} & {} & {} 
\\
{\Bobject_1}   \ar[rr,"{f_\Bobject}"{swap}]  & {} & {\Bobject_2}
\arrow[from=source,to=target,double,-implies,shorten <>=10pt,"{\varphi}"{description,pos=.5}]
\end{tikzcd}
\quad
\quad
\begin{tikzcd}[column sep=1em,row sep=.7em]
\varphi : 
{p_2\circ f_\Eobject}  \ar[rr, double, -implies] &&
{f_\Bobject\circ p_1}
\end{tikzcd}
$$
A morphism~(\ref{equation/lax-morphisms}) is called \emph{strict} when the 2-cell $\varphi$ is the identity.
We write in that case $(f_\Eobject,f_\Bobject)$ instead of $(f_\Eobject,f_\Bobject, \id{})$.
We also write $\ArrowStrict{\Ktwocategory}$ for the sub-2-category 
of $\ArrowLax{\Ktwocategory}$ of strict morphisms,
with the same notion of 2-cell.
It is essentially immediate that
\begin{proposition}\label{proposition/comprehension-structure}
A comprehension structure $(\comprehensionzero,\injectionzero)$
(in the sense of Def.~\ref{definition/comprehension-structure})
is the same thing as a morphism in $\Cat//\Cat$ of the form
\begin{equation}\label{equation/comprehension-structure}
\begin{tikzcd}[column sep=1em]
(f_{\Ecategory}, \id{\Bcategory}, \varphi) \quad : \quad (\Ecategory,\Bcategory,p:\Ecategory\to \Bcategory) \arrow[rr] 
&& (\Bcategory,\Bcategory,\id{\Bcategory}:\Bcategory\to \Bcategory)
\end{tikzcd}
\end{equation}
\end{proposition}
One main contribution of the paper is to revisit in this 2-categorical style
the hierarchy of comprehension categories described by Jacobs \cite{jacobs-paper}.
To that purpose, we introduce three corresponding levels of comprehension structures,
each of them coming with an elementary and concise 2-categorical formulation,
as depicted in the figure below:
\begin{itemize}
\item comprehension structures, Def.~\ref{definition/comprehension-structure} 
as reformulated in Prop.~\ref{proposition/comprehension-structure},
\item comprehension structures with section, Def.~\ref{definition/comprehension-structure-with-section}
as reformulated in Prop.~\ref{proposition/comprehension-structure-with-section},
\item comprehension structures with image, Def.~\ref{definition/comprehension-structure-with-image}
as formulated in Def.~\ref{definition/comprehension-structure-with-image-in-K}.
\end{itemize}
One basic observation is that our minimalist notion 
of comprehension structure (Def.~\ref{definition/comprehension-structure})
generalizes Jacobs' notion of comprehension category, by relaxing 
the assumption that the associated functor $\jacobscomprehension:\Ecategory\to\Barrowcategory$
in (\ref{equation/Barrowcategory-cod-functor}) transports every $p$-cartesian map of $\Ecategory$
to a $\codfunctor$-cartesian map of $\Barrowcategory$, that is, to a pullback diagram 
of the form~(\ref{equation/square-diagram-for-comprehension}) in the category $\Bcategory$.
This observation underlies the first layer (in dark green) of our classification below.
\begin{center}
{\includegraphics[height=9em]{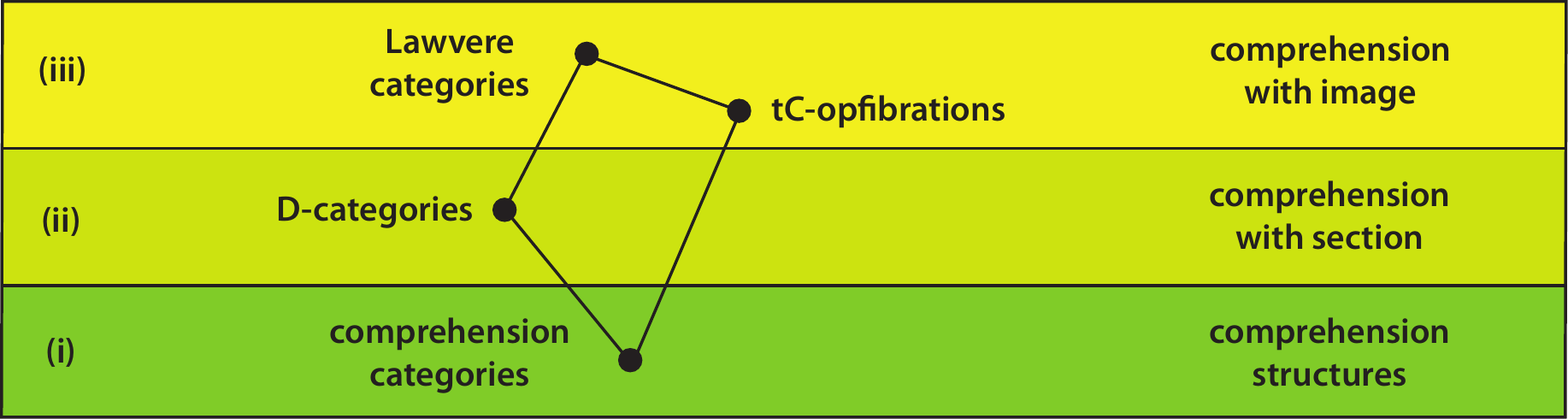}}
\end{center}
%

\medbreak

\noindent
\textbf{(ii) Comprehension structures with section.}
We move to the next layer and consider Ehrhard's notion of $D$-category~\cite{ehrhard,jacobs-paper}
which is based on a convenient but somewhat mysterious recipe
to equip a functor $p:\Ecategory\to\Bcategory$ with a comprehension structure $(\comprehensionzero,\injectionzero)$.
The recipe~\cite{ehrhard,jacobs-paper}
works in two stages: (1)~first, one equips the functor $p$ with a section $\smallstar:\Bcategory\to\Ecategory$, 
(2)~then one requires that the section $\smallstar$ has a right adjoint $\comprehensionzero:\Ecategory\to\Bcategory$.
Recall that a section $\smallstar:\Bcategory\to\Ecategory$ is a functor such that $p\circ\smallstar=\id{\Bcategory}$.
This leads us to the following definition:
\begin{definition}\label{definition/comprehension-structure-with-section}
A comprehension structure with section on a functor $p:\Ecategory\to\Bcategory$
is a section $\smallstar:\Bcategory\to\Ecategory$
together with a right adjoint functor $\comprehensionzero:\Ecategory\to\Bcategory$.
\end{definition}
One astonishing aspect of the definition is 
that the natural transformation $\injectionzero:\comprehensionzero\Rightarrow p$ 
of the associated comprehension structure $(\comprehensionzero,\iota)$ is not given explicitly,
but derived as the image by the functor $p:\Ecategory\to\Bcategory$ 
of the counit $\smallstar\circ\comprehensionzero\Rightarrow \id{\Ecategory}$ of the adjunction $\smallstar\dashv\comprehensionzero$.
From this it follows that the relationship between the natural transformation $\iota$ and
the two functors $\smallstar,\comprehensionzero$ is not entirely obvious from a conceptual point of view.
We clarify this point by observing here that the original adjunction $\smallstar\dashv\comprehensionzero$ 
of Def.~\ref{definition/comprehension-structure-with-section} living in $\Ktwocategory=\Cat$
is the ``emerged part'' of a more fundamental adjunction $\smallstar\dashv(\comprehensionzero,\injectionzero)$
living in the 2-category $\ArrowLax{\Cat}$, and where the natural transformation $\iota$ is thus integrated.
A preliminary observation is that
\begin{proposition}\label{proposition/section}
A section of the functor $p:\Ecategory\to\Bcategory$
is the same thing as a strict morphism in $\Cat/\Cat$ of the form
\begin{equation}\label{equation/section}
\begin{tikzcd}[column sep=1em]
(s_{\Ecategory}, \id{\Bcategory}) \quad : \quad  (\Bcategory,\Bcategory,\id{\Bcategory}:\Bcategory\to \Bcategory) \arrow[rr]
&& (\Ecategory,\Bcategory,p:\Ecategory\to \Bcategory) 
\end{tikzcd}
\end{equation}
\end{proposition}
We will prove in the course of the paper (see \S\ref{section/comprehension-with-section},
Prop.~\ref{proposition/comprehension-structure-with-section-in-K}) 
that the adjunction $\smallstar\dashv\comprehensionzero$ 
in Def.~\ref{definition/comprehension-structure-with-section}
may be equivalently formulated as an adjunction in $\ArrowLax{\Cat}$
between the section $\smallstar$ seen as a strict morphism~(\ref{equation/section})
and the comprehension structure $(\comprehensionzero,\iota)$
seen as a morphism~(\ref{equation/comprehension-structure}).
%
This property establishes the secretly 2-categorical nature 
of the notion (Def.~\ref{definition/comprehension-structure-with-section})
of comprehension structure with section:
%
\begin{proposition}\label{proposition/comprehension-structure-with-section}
A comprehension structure with section is a comprehension structure~(\ref{equation/comprehension-structure})
right adjoint to a section~(\ref{equation/section}) in the 2-category $\ArrowLax{\Cat}$.
\end{proposition}
The resulting 2-categorical notion of \emph{comprehension structure with section}
(Prop. \ref{proposition/comprehension-structure-with-section}) captures
the essence of the notion of $D$-category, and generalizes it
in an interesting and useful way to the categories of algebras and coalgebras,
see \S\ref{section/comprehension-with-section} and \S\ref{section/application} for a discussion.
%

\medbreak

\noindent
\textbf{(iii) Comprehension structures with image.}
We move finally to the next layer of our hierarchy,
and observe that the functor $p:\Ecategory\to\Bcategory$ is required to be an opfibration
in both notions of Lawvere category and of $tC$-opfibration \cite{jacobs-paper,fumex-phdthesis}.
From this follows that the functor $p:\Ecategory\to\Bcategory$ has an image structure,
in the sense elaborated in \S\ref{section/functors-image-structure} of this paper.
%
%
This additional image structure on the functor $p$ enables one to construct a functor 
\begin{equation}\label{equation/image-functor}
\begin{tikzcd}[column sep =1em]
\imagefunctor \quad : \quad \Barrowcategory \arrow[rr] && \Ecategory
\end{tikzcd}
\end{equation}
from the arrow category $\Barrowcategory$ of the basis category $\Bcategory$
to the category~$\Ecategory$.
The functor $\imagefunctor$ transports every morphism $f:A\to B$ of the basis category~$\Bcategory$
to an object $\imagefunctor(f)$ in the fiber category~$\Ecategory_B$ of the object~$B$,
called the image of $f:A\to B$, and satisfying the expected universality property,
see \S\ref{section/functors-image-structure} for details.
%
By construction, the image functor (\ref{equation/image-functor})
makes the diagram below commute:
\begin{equation}\label{equation/Barrowcategory-image-functor}
\begin{tikzcd}[column sep = .8em, row sep = .8em]
\Ecategory\arrow[rrdd,"p"{swap}] &&&& \Barrowcategory\arrow[lldd,"{\codfunctor}"]\arrow[llll,"{\imagefunctor}"{swap}]
\\
\\
&& \Bcategory
\end{tikzcd}
\end{equation}
In order to recover a comprehension structure (\ref{equation/Barrowcategory-cod-functor}),
the definition of a Lawvere category requires that the functor $\imagefunctor$
has a right adjoint $\jacobscomprehension:\Ecategory\to\Barrowcategory$ 
in the fibered sense above the category~$\Bcategory$.
This leads us to the following definition.
\begin{definition}\label{definition/comprehension-structure-with-image}
A comprehension structure with image on a functor $p:\Ecategory\to\Bcategory$
is a functor $\imagefunctor:\Barrowcategory\to\Ecategory$
together with a right adjoint $\jacobscomprehension:\Ecategory\to \Barrowcategory$
in the fibered sense above $\Bcategory$.
\end{definition}
On the other hand, and somewhat surprisingly,
the definition of $tC$-opfibration is apparently weaker,
since it only requires that the functor $p:\Ecategory\to\Bcategory$ 
has a comprehension structure with section, in the sense of Def.~\ref{definition/comprehension-structure-with-section}.
In order to clarify the situation,
and to get a clean and harmonious picture,
we establish that every $tC$-opfibration has comprehension with image
(in the sense of Def~\ref{definition/comprehension-structure-with-image})
using the following statement, which applies in particular
to the case of an opfibration $p:\Ecategory\to\Bcategory$:
\begin{proposition}\label{proposition/comprehension-structure-with-image-is-imagreand-comprehension-structure-with-section} 
Suppose that the functor $p:\Ecategory\to\Bcategory$ has an image structure.
In that case, every comprehension structure with section (in the sense of Def. \ref{definition/comprehension-structure-with-section})
defines a comprehension structure with image (in the sense of Def.~\ref{definition/comprehension-structure-with-image}).
\end{proposition}
\paragraph*{Illustration: inductive reasoning on algebras, coinductive reasoning on coalgebras}
Suppose given a functor $p:\Ecategory\to\Bcategory$ equipped with a comprehension structure 
with section $\smallstar:\Bcategory\to\Ecategory$, where the categories 
$\Ecategory$ and $\Bcategory$ are moreover equipped with endofunctors 
$\Ffunctor:\Bcategory\to\Bcategory$ and $\Gfunctor:\Ecategory\to\Ecategory$ 
related by a distributivity law
\begin{equation}\label{equation/distributivity-law}
\begin{tikzcd}[column sep = .8em, row sep = .8em]
\delta \quad : \quad \Ffunctor\circ p \arrow[rr,double,-implies] && p\circ \Gfunctor
\quad : \quad \Ecategory\arrow[rr] && \Bcategory.
\end{tikzcd}
\end{equation}
One guiding ambition of our 2-categorical account of comprehension structures
is to explain by conceptual means the recent characterization by Fumex, Ghani and Johann 
\cite{fumex-phdthesis,ghani-johann-fumex} of the initial $\Gfunctor$-algebra of~$\Ecategory$
as the section $\smallstar_{A}$ of the initial $\Ffunctor$-algebra $\mu F$
of the basis category $\Bcategory$.
To that purpose, we describe in \S\ref{section/application} the necessary and sufficient conditions 
which characterize when the distributivity law (\ref{equation/distributivity-law}) on a comprehension structure 
with section $p:\Ecategory\to\Bcategory$ induces a comprehension structure with section
$\Alg{}(p):\Alg{\Gfunctor}(\Ecategory)\to\Alg{\Ffunctor}(\Bcategory)$ on the associated categories of algebras.
In this situation, we obtain a simple conceptual explanation for the forementioned result (\cite{ghani-johann-fumex}, Thm 2.10)
by Fumex, Ghani and Johann:
\begin{corollary}\label{corollary/comprehension-structures}
The comprehension structure with section $\smallstar:\Bcategory\to\Ecategory$
lifts to a comprehension structure with section $\smallstar:\Alg{\Ffunctor}(\Bcategory)\to\Alg{\Gfunctor}(\Ecategory)$
which is left adjoint to comprehension $\comprehensionzero$
and thus transports the initial $\Ffunctor$-algebra $\mu \Ffunctor$ to the initial $\Gfunctor$-algebra $\mu \Gfunctor=\smallstar_{\mu \Ffunctor}$.
\end{corollary}
We proceed dually in the case of quotient structures and obtain necessary
and sufficient conditions to ensure that 
\begin{corollary}\label{corollary/quotient-structures}
The quotient structure with section $\smallstar:\Bcategory\to\Ecategory$ 
lifts to a quotient structure with section $\smallstar:\CoAlg{\Ffunctor}(\Bcategory)\to\CoAlg{\Gfunctor}(\Ecategory)$
which is right adjoint to quotient $\quotient{-}$ and thus transports the terminal $\Ffunctor$-coalgebra 
$\nu \Ffunctor$ to the terminal $\Gfunctor$-coalgebra $\nu \Gfunctor=\smallstar_{\nu \Ffunctor}$.
\end{corollary}

\subsection*{Plan of the paper}
After this long and detailed introduction, we recall in \S\ref{section/arrow-two-categories} the notion of arrow 2-category
$\Ktwocategory//\Ktwocategory$ and establish in \S\ref{section/formal-adjunctions}
a simple and useful description of the formal adjunctions in this 2-category.
This leads us to formulate in \S\ref{section/comprehension-with-section}
our 2-categorical notion of \emph{comprehension with section}.
We then formulate in~\S\ref{section/path-objects} and \S\ref{section/functors-image-structure}
the notion of path object $(\Bobjectarrow,\beta)$ of an object $\Bobject$ in any 2-category~$\Ktwocategory$,
and the related notion of morphism $p:\Eobject\to\Bobject$ with an image structure.
This leads us to establish in \S\ref{section/comprehension-with-image} that
a comprehension with image $p:\Eobject\to\Bobject$ is the same thing 
as a comprehension structure with section $\smallstar:\Bobject\to\Eobject$,
whose underlying section comes with an image structure,
and thus a morphism $\imagefunctor:\Bobjectarrow\to\Eobject$.
We then illustrate in \S\ref{section/application-to-induction-and-coinduction}
the benefits of our 2-categorical approach with the example of inductive
and coinductive reasoning on algebra and coalgebra structures, 
and finally conclude in \S\ref{section/conclusion}.

\section{Definition of the arrow 2-categories $\ArrowLax{\Ktwocategory}$ and $\ArrowStrict{\Ktwocategory}$}\label{section/arrow-two-categories}
We explained in the introduction, see~(\ref{equation/lax-morphisms}),
how to define the objects and the morphisms
of the 2-category $\ArrowLax{\Ktwocategory}$ associated to a 2-category $\Ktwocategory$.
For the sake of completeness, we recall now that a 2-cell 
$$\begin{tikzcd}[column sep = 1em]
(\theta_{\Bobject}, \theta_{\Eobject}) \,\, : \,\,
\Triple{f_{\Eobject}}{\varphi}{f_{\Bobject}} \arrow[rr,-implies,double]
&&
\Triple{g_{\Eobject}}{\psi}{g_{\Bobject}}\,\, : \,\,
\Triple{\Eobject_1}{p_1}{\Bobject_1} \arrow[rr] 
&&
\Triple{\Eobject_2}{p_2}{\Bobject_2} 
\end{tikzcd}$$
of the 2-category $\ArrowLax{\Ktwocategory}$ is defined as a pair 
of 2-cells $\theta_B:f_{\Bobject}\Rightarrow g_{\Bobject}$ and $\theta_\Eobject:f_{\Eobject}\Rightarrow g_{\Eobject}$ 
of the original 2-category~$\Ktwocategory$, making the two pasting diagrams equal:
$$
\begin{tikzcd}[column sep=1em,row sep=.8em]
{\Eobject_1}
\arrow[rrrr,bend left=40,"{f_{\Eobject}}", ""{name=source-up,swap}]
\arrow[dddd,"{p_1}"{swap},""{name=target-up}]
&&&&
{\Eobject_2}
\arrow[dddd,"{p_2}",""{name=source-up,swap}]
\\
\\
\\
\\
{\Bobject_1}
\arrow[rrrr,bend left=40,"{f_{\Bobject}}", ""{name=source-down,swap}]
\arrow[rrrr,bend right=40,"{g_{\Bobject}}"{swap},""{name=target-down}]
&&&&
{\Bobject_2}
\arrow[from=source-down,to=target-down,double,-implies,"{\theta_B}"{description,pos=.5}]
\arrow[from=source-up,to=target-up,bend right=40,double,-implies,"{\varphi}"{swap,pos=.5}]
\end{tikzcd}
\quad\quad = \quad\quad
\begin{tikzcd}[column sep=1em,row sep=.8em]
{\Eobject_1}
\arrow[rrrr,bend left=40,"{f_{\Eobject}}", ""{name=source-up,swap}]
\arrow[rrrr,bend right=40,"{g_{\Eobject}}"{swap},""{name=target-up}]
\arrow[dddd,"{p_1}"{swap},""{name=target-down}]
&&&&
{\Eobject_2}
\arrow[dddd,"{p_2}",""{name=source-down,swap}]
\\
\\
\\
\\
{\Bobject_1}
\arrow[rrrr,bend right=40,"{g_{\Bobject}}"{swap}]
&&&&
{\Bobject_2}
\arrow[from=source-up,to=target-up,double,-implies,"{\theta_\Eobject}"{description,pos=.5}]
\arrow[from=source-down,to=target-down,bend left=40,double,-implies,"{\psi}"{pos=.5}]
\end{tikzcd}
$$
It is worth mentioning that, thanks to this carefully chosen definition of 2-cells,
there exists a pair of 2-functors 
\begin{equation}\label{equation/source-target}
\begin{tikzcd}[column sep = 1.5em]
\ArrowLax{\Ktwocategory} \arrow[rrrr,yshift=.4em,"{\source}"] \arrow[rrrr,yshift=-.4em,"{\target}"{swap}]
&&&&
\Ktwocategory
\end{tikzcd}
\end{equation}
defined as the expected first and second projections,
which transport every object $(\Eobject,\Bobject,p:\Eobject\to\Bobject)$ of the 2-category $\ArrowLax{\Ktwocategory}$
to the object
$$\source(\Eobject,\Bobject,p:\Eobject\to\Bobject)=\Eobject \quad\quad \quad\quad 
\target(\Eobject,\Bobject,p:\Eobject\to\Bobject)=\Bobject$$
of the underlying 2-category~$\Ktwocategory$.
Finally, let us also mention that the 2-category $\ArrowStrict{\Ktwocategory}$ of strict morphisms
in $\ArrowLax{\Ktwocategory}$ comes exactly with the same notion of 2-cell.
In other words, the inclusion 2-functor $\ArrowStrict{\Ktwocategory}\to\ArrowLax{\Ktwocategory}$
is locally fully faithful.

\section{Formal adjunctions in the 2-category $\ArrowLax{\Ktwocategory}$}\label{section/formal-adjunctions}
As explained in the introduction in the case $\Ktwocategory=\Cat$,
one main observation of the paper is that the notion of \emph{comprehension structure with section}
(Def. \ref{definition/comprehension-structure-with-section}) can be elegantly
expressed as a specific form of adjunction living in the 2-category $\ArrowLax{\Ktwocategory}$
(Prop. \ref{proposition/comprehension-structure-with-section}).
As a warm up exercise, we study the notion of formal adjunction in $\ArrowLax{\Ktwocategory}$
in the sense of Street~\cite{street:formal-monads} and relate it in full generality
to the notion of formal adjunction in the original 2-category $\Ktwocategory$.
Suppose given a pair of morphisms 
\begin{equation}\label{equation/LR}
\begin{array}{c}
\begin{tikzcd}[column sep=1em]
L \, = \, (L_\Eobject,L_\Bobject, \varphi) \quad : \quad (\Eobject_1,\Bobject_1,p_1:\Eobject_1\to \Bobject_1) \arrow[rr] && (\Eobject_2,\Bobject_2,p_2:\Eobject_2\to \Bobject_2)
\end{tikzcd}
\\
\begin{tikzcd}[column sep=1em]
R \, = \, (R_\Eobject,R_\Bobject, \psi) \quad : \quad (\Eobject_2,\Bobject_2,p_2:\Eobject_2\to \Bobject_2) \arrow[rr] && (\Eobject_1,\Bobject_1,p_1:\Eobject_1\to \Bobject_1)
\end{tikzcd}
\end{array}
\end{equation}
living in the 2-category $\ArrowLax{\Ktwocategory}$, and thus depicted
as below in the underlying 2-category $\Ktwocategory$:
$$
\begin{array}{ccc}
\begin{tikzcd}[column sep=3em,row sep=.5em]
{\Eobject_1}  \ar[dd, "{p_1}"{swap},"{}"{name=target}] \ar[rr,"{L_\Eobject}"] 
&& 
{\Eobject_2} \ar[dd,"{p_2}","{}"{swap,name=source}] 
\\
{} & {} & {} 
\\
{\Bobject_1}   \ar[rr,"{L_\Bobject}"{swap}]  & {} & {\Bobject_2}
\arrow[from=source,to=target,double,-implies,shorten <>=10pt,"{\varphi}"{description,pos=.5}]
\end{tikzcd}
&
\quad\quad\quad
&
\begin{tikzcd}[column sep=3em,row sep=.5em]
{\Eobject_1}  \ar[dd, "{p_1}"{swap},"{}"{name=source}] 
&& 
{\Eobject_2} \ar[dd,"{p_2}","{}"{swap,name=target}] \ar[ll,"{R_\Eobject}"{swap}] 
\\
{} & {} & {} 
\\
{\Bobject_1}  & {} & {\Bobject_2}  \ar[ll,"{R_\Bobject}"] 
\arrow[from=source,to=target,double,-implies,shorten <>=10pt,"{\psi}"{description,pos=.5}]
\end{tikzcd}
\end{array}
$$
By definition, a \emph{formal adjunction} between $L$ and $R$ 
in the 2-category $\ArrowLax{\Ktwocategory}$ is defined as a pair of 2-cells
$$\begin{small}
\begin{array}{c}
\begin{tikzcd}[column sep = .55em]
(\eta_\Bobject, \eta_\Eobject) \, : \,
\Triple{\idzero{\Eobject_1}}{\idzero{}}{\idzero{\Bobject_1}} 
\arrow[rr,-implies,double]
&&
\Triple{R_\Eobject\circ L_\Eobject}{(R_\Bobject\circ\varphi)(\psi\circ L_\Eobject)}{R_\Bobject\circ L_\Bobject} 
\, : \,
\Triple{\Eobject_1}{p_1}{\Bobject_1} 
\arrow[rr]
&&
\Triple{\Eobject_1}{p_1}{\Bobject_1}
\end{tikzcd}
\\
\begin{tikzcd}[column sep = .55em]
(\varepsilon_\Bobject, \varepsilon_\Eobject) \, : \,
\Triple{L_\Eobject\circ R_\Eobject}{(L_\Bobject\circ\psi)(\varphi\circ R_\Eobject)}{L_\Bobject\circ R_\Bobject} 
\arrow[rr,-implies,double]
&&
\Triple{\idzero{\Eobject_2}}{\idzero{}}{\idzero{\Bobject_2}} 
\, : \,
\Triple{\Eobject_2}{p_2}{\Bobject_2} 
\arrow[rr]
&&
\Triple{\Eobject_2}{p_2}{\Bobject_2}
\end{tikzcd}
\end{array}
\end{small}$$
in the 2-category $\ArrowLax{\Ktwocategory}$, 
satisfying the triangular equations, see \cite{street:formal-monads,Mellies09panorama} for details.
One nice consequence of this definition by generators and relations 
is that the resulting notion of formal adjunction is preserved by 2-functors.
Every formal adjunction $L\dashv R$ in $\ArrowLax{\Ktwocategory}$
is thus transported by the 2-functors~(\ref{equation/source-target}) into a pair of formal adjunctions
$L_\Bobject\dashv R_\Bobject$ and $L_\Eobject\dashv R_\Eobject$ in the underlying 2-category~$\Ktwocategory$.
From this follows that the 2-cell $\psi:p_1\circ R_\Eobject\Rightarrow R_\Bobject\circ p_2$ in $\Ktwocategory$ 
induces a 2-cell $\widetilde{\psi}:L_\Bobject\circ p_1\Rightarrow p_2\circ L_\Eobject$ 
called the \emph{mate} of $\psi$, of the form below:
$$
\begin{tikzcd}[column sep=3em,row sep=.4em]
{\Eobject_1}  \ar[dd, "{p_1}"{swap},"{}"{name=source}] \ar[rr,"{L_\Eobject}"] 
&& 
{\Eobject_2} \ar[dd,"{p_2}","{}"{swap,name=target}] 
\\
{} & {} & {} 
\\
{\Bobject_1}  \ar[rr,"{L_\Bobject}"{swap}] & {} & {\Bobject_2}  
\arrow[from=source,to=target,double,-implies,shorten <>=10pt,"{\widetilde{\psi}}"{description,pos=.5}]
\end{tikzcd}
$$
%
Suppose given two morphisms $L$ and $R$ in the 2-category $\ArrowLax{\Ktwocategory}$ as in~(\ref{equation/LR}).
In that case,
\begin{proposition}\label{proposition/characterization-of-adjunctions}
A formal adjunction $L\dashv R$ in the 2-category $\ArrowLax{\Ktwocategory}$
is the same thing as a pair of formal adjunctions $L_\Bobject\dashv R_\Bobject$ and $L_\Eobject\dashv R_\Eobject$
in the 2-category $\Ktwocategory$, such that the induced mate $\widetilde{\psi}$ 
of the 2-cell $\psi$ is the inverse of the 2-cell $\varphi$ in the 2-category $\Ktwocategory$.
\end{proposition}
From this follows easily that 
\begin{proposition}\label{proposition/lifting}
A pair of formal adjunctions $L_\Bobject\dashv R_\Bobject$ and $L_\Eobject\dashv R_\Eobject$
in the 2-category~$\Ktwocategory$ lifts to a formal adjunction 
$\Triple{L_\Bobject}{\varphi}{L_\Eobject}\dashv\Triple{R_\Bobject}{\psi}{R_\Eobject}$
in the 2-category~$\ArrowLax{\Ktwocategory}$ precisely when 
the 2-cell $\varphi$ is invertible in $\Ktwocategory$
and the 2-cell $\psi$ coincides with the mate of $\varphi^{-1}$.
\end{proposition}
%
It should be mentioned that a similar observation is made by Kelly~\cite{kelly1974doctrinal} (Prop. 1.3)
on the 2-category of $D$-algebras derived from a 2-monad $D$ on the 2-category $\Ktwocategory$, 
see also \cite{streetReview}, Section 3.5.
It should be also noted that the characterization of formal adjunctions in $\ArrowLax{\Ktwocategory}$
is also very similar to the description of formal adjunctions in the 2-category of monoidal categories
and lax monoidal functors, see for instance \cite{Mellies09panorama}.

\section{Comprehension structures with section}\label{section/comprehension-with-section}
In this section, we suppose given a morphism $p:\Eobject\to\Bobject$ in the 2-category $\Ktwocategory$
and establish (Prop.~\ref{proposition/comprehension-structure-with-section-in-K})
that a \emph{comprehension structure with section} 
on a morphism $p:\Eobject\to\Bobject$ originally defined (Def.~\ref{definition/comprehension-structure-with-section-in-K})
as an adjunction $\smallstar\dashv\comprehensionzero$ in the 2-category $\Ktwocategory$,
can be in fact lifted (and thus equivalently defined) as a specific form of adjunction
$\smallstar\dashv(\comprehensionzero,\injectionzero)$ in the 2-category $\ArrowLax{\Ktwocategory}$.
As expected, the proof of Prop.~\ref{proposition/comprehension-structure-with-section-in-K}
relies on the characterization of formal adjunctions in $\ArrowLax{\Ktwocategory}$ just established
in the previous section, see Prop.~\ref{proposition/lifting}.
The result provides a general 2-categorical formulation of 
the proposition (Prop.~\ref{proposition/comprehension-structure-with-section})
stated in the introduction for the particular case $\Ktwocategory=\Cat$.
In order to perform our construction at this general 2-categorical level of abstraction,
we start by defining a \emph{comprehension structure} on a morphism $p:\Eobject\to\Bobject$ 
in the 2-category~$\Ktwocategory$, in a way which generalizes
what we did in the introduction (see Def. \ref{definition/comprehension-structure}) 
in the specific case $\Ktwocategory=\Cat$.
\begin{definition}\label{definition/comprehension-structure-general-K}
A comprehension structure on the morphism $p:\Eobject\to\Bobject$ 
is a pair $(\comprehensionzero,\injectionzero)$ 
consisting of a morphism $\comprehensionzero:\Eobject\to\Bobject$
and of a 2-cell $\injectionzero:\comprehensionzero\Rightarrow p$
in the 2-category~$\Ktwocategory$.
\end{definition}
We carry on as we did in the introduction (see Prop.~\ref{proposition/comprehension-structure})
and observe that
\begin{proposition}\label{proposition/comprehension-structure-in-K}
A comprehension structure $(\comprehensionzero,\injectionzero)$ on $p:\Eobject\to\Bobject$
is the same thing as a morphism in the 2-category $\ArrowLax{\Ktwocategory}$
of the form 
\begin{equation}\label{equation/comprehension-structure-in-K}
\begin{tikzcd}[column sep=1em]
(f_{\Eobject}, \id{\Bobject}, \varphi) \quad : \quad (\Eobject,\Bobject,p:\Eobject\to\Bobject) \arrow[rr] && (\Bobject,\Bobject,\id{\Bobject}:\Bobject\to \Bobject)
\end{tikzcd}
\end{equation}
\end{proposition}
We proceed in just the same way as we did in the introduction with Prop.~\ref{proposition/section},
and characterize a section $\smallstar:\Bobject\to\Eobject$ of the morphism $p:\Eobject\to\Bobject$
as a specific form of strict morphism:
%
\begin{proposition}\label{proposition/section-in-K}
A section $\smallstar:\Bobject\to\Eobject$ of the morphism $p:\Eobject\to\Bobject$ in $\Ktwocategory$
is the same thing as a strict morphism in $\ArrowLax{\Ktwocategory}$ 
of the form
\begin{equation}\label{equation/section-in-K}
\begin{tikzcd}[column sep=1em]
(s_\Eobject, \id{\Bobject}) \quad : \quad  (\Bobject,\Bobject,\id{\Bobject}:\Bobject\to \Bobject)
\arrow[rr] && (\Eobject,\Bobject,p:\Eobject\to\Bobject) 
\end{tikzcd}
\end{equation}
\end{proposition}
We are now ready to give our general 2-categorical definition
of \emph{comprehension structure with section}
on the morphism $p:\Eobject\to\Bobject$ in the 2-category $\Ktwocategory$.
\begin{definition}\label{definition/comprehension-structure-with-section-in-K}
A comprehension structure with section on $p:\Eobject\to\Bobject$ is a section $\smallstar:\Bobject\to\Eobject$
together with a right adjoint $\comprehensionzero:\Eobject\to \Bobject$ in the 2-category $\Ktwocategory$.
\end{definition}
We then take advantage of Prop.~\ref{proposition/lifting} in order to establish that:
\begin{proposition}\label{proposition/comprehension-structure-with-section-in-K}
A comprehension structure with section 
is a comprehension structure~(\ref{equation/comprehension-structure-in-K})
right adjoint to a section~(\ref{equation/section-in-K}) in the 2-category $\ArrowLax{\Ktwocategory}$.
\end{proposition}
\begin{proof} Suppose given a section $\smallstar:\Bobject\to\Eobject$ of the morphism $p:\Eobject\to\Bobject$
in the 2-category~$\Ktwocategory$, described (Prop.~\ref{proposition/section-in-K})
as a strict morphism
$$
\begin{tikzcd}[column sep=1em]
L \, = \, (\smallstar,\id{\Bobject}, \id{}) \quad : \quad (\Bobject,\Bobject,\id{\Bobject}:\Bobject\to \Bobject)
\arrow[rr] && (\Eobject,\Bobject,p:\Eobject\to\Bobject)
\end{tikzcd}
$$
in the 2-category $\ArrowStrict{\Ktwocategory}$.
Consider moreover a morphism of the form
$$
\begin{tikzcd}[column sep=1em]
R \, = \, (\comprehensionzero,\id{\Bobject}, \iota) \quad : \quad 
(\Eobject,\Bobject,p:\Eobject\to\Bobject) \arrow[rr] && (\Bobject,\Bobject,\id{\Bobject}:\Bobject\to \Bobject)
\end{tikzcd}
$$
in the 2-category $\ArrowLax{\Ktwocategory}$.
The morphisms $L$ and $R$ of the 2-category $\ArrowLax{\Ktwocategory}$
can be depicted as follows in the underlying 2-category $\Ktwocategory$:
$$
\begin{array}{ccc}
\begin{tikzcd}[column sep=3em,row sep=.5em]
{\Bobject}  \ar[dd, "{\id{\Bobject}}"{swap},"{}"{name=target}] \ar[rr,"{\smallstar}"] 
&& 
{\Eobject} \ar[dd,"{p}","{}"{swap,name=source}] 
\\
{} & {} & {} 
\\
{\Bobject}   \ar[rr,"{\id{\Bobject}}"{swap}]  & {} & {\Bobject}
\end{tikzcd}
&
\quad\quad\quad
&
\begin{tikzcd}[column sep=3em,row sep=.5em]
{\Bobject}  \ar[dd, "{\id{\Bobject}}"{swap},"{}"{name=source}] 
&& 
{\Eobject} \ar[dd,"{p}","{}"{swap,name=target}] \ar[ll,"{\comprehensionzero}"{swap}] 
\\
{} & {} & {} 
\\
{\Bobject}  & {} & {\Bobject}  \ar[ll,"{\id{\Bobject}}"] 
\arrow[from=source,to=target,double,-implies,shorten <>=10pt,"{\iota}"{description,pos=.5}]
\end{tikzcd}
\end{array}
$$
Now, suppose that $\smallstar$ and $\comprehensionzero$
define a comprehension structure with section, 
in the sense of Def~\ref{definition/comprehension-structure-with-section-in-K}.
By definition, this means that there is an adjunction $\smallstar\dashv\comprehensionzero$.
By Prop.~\ref{proposition/lifting}, the adjunction $\smallstar\dashv\comprehensionzero$
lifts to an adjunction $L\dashv R$ between the morphisms $L$ and $R$ 
in the 2-category $\ArrowLax{\Ktwocategory}$ precisely when the 2-cell 
$\iota:\comprehensionzero\Rightarrow p$  in the 2-category $\Ktwocategory$ is the mate 
defined as $\iota = p\circ \varepsilon$, of the identity 2-cell $\id{}:\id{\Bobject}\Rightarrow p\circ\smallstar$.
Here, the 2-cell $\varepsilon:\comprehensionzero\circ\smallstar\Rightarrow\id{}$
denotes the counit of the adjunction $\smallstar\dashv\comprehensionzero$
in the 2-category $\Ktwocategory$.
This establishes one direction of the proof, while the other direction is immediate.
\end{proof}

\medbreak
\noindent
Note that the definition of the 2-cell $\iota:\comprehensionzero\Rightarrow p$
as the mate $\iota:\comprehensionzero\Rightarrow p$ of the identity 2-cell $\id{}:\id{\Bobject}\Rightarrow p\circ\smallstar$
provides a conceptual explanation for the definition of $\iota:\comprehensionzero\Rightarrow p$
as the image by $p$ of the counit $\varepsilon$ of the counit of the adjunction $\smallstar\dashv\comprehensionzero$.


\section{Path objects in a 2-category}\label{section/path-objects}
We introduce the notion of \emph{path object} on an object~$\Bobject$ in a 2-category~$\Ktwocategory$.
This construction, which generalizes the usual construction of the arrow category $\Barrowcategory$
of a category $\Bcategory$ when $\Ktwocategory=\Cat$, will play an important role
in~\S\ref{section/application-to-induction-and-coinduction}
when we apply our constructions to inductive and coinductive reasoning on initial algebras
and terminal algebras.
Every object $\Bobject$ in a 2-category $\Ktwocategory$ induces a 2-functor
\begin{equation}\label{equation/Karrowtwofunctor}
\begin{tikzcd}[column sep=1em]
{{\Ktwocategory(-,\Bobject)}^{\to}} \,\, : \,\, \Ktwocategory \arrow[rr] && \Cat
\end{tikzcd}
\end{equation}
which transports every object $X\in\Ktwocategory$ to the arrow category ${\Ktwocategory(\Xobject,\Bobject)}^{\to}$
of the hom-category $\Ktwocategory(\Xobject,\Bobject)$ between $\Xobject$ and $\Bobject$.
\begin{definition}\label{definition/path-object}
A path object for the object $\Bobject$ in the 2-category $\Ktwocategory$
is a pair $(\Bobject^{\to},\beta)$ consisting of an object $\Bobject^{\to}$ and of
a family of isomorphisms
$$
\beta_\Xobject \,\, : \,\, \Ktwocategory(\Xobject,\Bobject^{\to}) \cong \Ktwocategory(\Xobject,\Bobject)^{\to}
$$
2-natural in the object $X$.
Terminology: one says in that case that the pair $(\Bobject^{\to},\beta)$
defines a representation of the 2-functor~(\ref{equation/Karrowtwofunctor}).
\end{definition}
Note that every path object $(\Bobject^{\to},\beta)$ comes equipped with three morphisms and a 2-cell
$$
\begin{array}{ccccc}
\begin{tikzcd}[column sep=.5em]
\arrowcod,\arrowdom: \Bobjectarrow \arrow[rr] && \Bobject
\end{tikzcd}
&&
\begin{tikzcd}[column sep=.5em]
\arrowid: \Bobject \arrow[rr] && \Bobjectarrow
\end{tikzcd}
&&
\begin{tikzcd}[column sep=.5em]
\arrowcell: \arrowdom \arrow[rr,double,-implies] && \arrowcod : \Bobjectarrow \arrow[rr] && \Bobject
\end{tikzcd}
\end{array}
$$
satisfying $\arrowdom\circ\arrowid=\id{\Bobject}=\arrowcod\circ\arrowid$
and that the 2-cell $\arrowcell\circ\arrowid$ coincides with the identity 2-cell on $\id{\Bobject}$.
A remarkable property is that 
\begin{proposition}
The three morphisms $\arrowcod$, $\arrowid$ and $\arrowdom$
are related by a pair of formal adjunctions $\arrowcod \dashv \arrowid$
and $\arrowid \dashv \arrowdom$ in the 2-category $\Ktwocategory$.
\end{proposition}
Going back to the definition (Def.~\ref{definition/comprehension-structure-with-section-in-K})
of a comprehension structure with section, this establishes that 
\begin{proposition}
Every path object $(\Bobject^{\to},\beta)$ defines a comprehension structure with section $\arrowid:\Bobject\to \Bobject^{\to}$
on the morphism $\arrowcod:\Bobject^{\to}\to \Bobject$ in the 2-category $\Ktwocategory$.
\end{proposition}
Note that the comprehension structure $(\comprehensionzero,\injectionzero)$
constructed in Prop.~\ref{proposition/comprehension-structure-with-section-in-K}
using a 2-categorical mate in $\Ktwocategory$
is provided in that case by the pair $(\arrowdom,\arrowcell)$
with 2-cell $\arrowcell:\arrowdom\Rightarrow\arrowcod$.

\section{Functors with image structure}\label{section/functors-image-structure}
In this section, we introduce the notion of \emph{functor $p:\Ecategory\to\Bcategory$ with image structure}
which weakens (and thus generalizes) the usual notion of Grothendieck opfibration $p:\Ecategory\to\Bcategory$.
\begin{definition}[Image structure]\label{definition/imagestructureonCat}
An image structure on a functor $p:\Ecategory\to\Bcategory$ 
is a pair $(\smallstar,\Pushmap{})$ consisting of a section $\smallstar$
defined as a functor $\smallstar:\Bcategory\to\Ecategory$ 
satisfying the equation $p\circ\smallstar=\id{\Bcategory}$,
together with a family~$\lambda$ of opcartesian morphisms
\begin{equation}\label{equation/image-structure}
\begin{tikzcd}[column sep=.8em]
u \,\, : \,\, A \arrow[rr] && B
\quad \quad \models \quad \quad
{\Pushmap{u}} \,\, : \,\, {\smallstar_A} \arrow[rr] && \push{u}{A}
\end{tikzcd}
\end{equation}
indexed by the morphisms $u:A\to B$ of the basis category $\Bcategory$.
We will also suppose for convenience that the opcartesian morphism
$$
\begin{tikzcd}[column sep=.8em]
\id{A} \,\, : \,\, A \arrow[rr] && A
\quad \quad \models \quad \quad
{\Pushmap{u}} \,\, : \,\, {\smallstar_A} \arrow[rr] && \push{\id{A}}{A}
\end{tikzcd}
$$
coincides with the identity morphism on $\smallstar_A$.
\end{definition}
Here, we follow the fibered philosophy of refinement type systems~\cite{MelliesZeilberger15}, and write 
$$
\begin{tikzcd}[column sep=.8em]
u \,\, : \,\, A \arrow[rr] && B
\quad \quad \models \quad \quad
f \,\, : \,\, R \arrow[rr] && S
\end{tikzcd}
$$
when a morphism $f:R\to S$ in the category $\Ecategory$ has image $p(f)=u:A\to B$ in the category~$\Bcategory$.
The intuition is that the morphism $f:R\to S$ is ``above'' the morphism $u:A\to B$, and ``dependent'' of it.
Accordingly, we write $\Ecategory_{u:A\to B}(R,S)$ for the set of such morphisms $f:R\to S$ such that $p(f)=u$.
Let us recall what universal property is required of the morphism~(\ref{equation/image-structure})
in order to make it opcartesian.
By precomposition in the category~$\Ecategory$, every morphism
$$
\begin{tikzcd}[column sep=.8em]
v \,\, : \,\, B \arrow[rr] && B'
\quad \quad \models \quad \quad
h \,\, : \,\, \push{u}{A} \arrow[rr] && S'
\end{tikzcd}
$$
induces a morphism
$$
\begin{tikzcd}[column sep=.8em]
v\circ u \,\, : \,\, A \arrow[rr] && B'
\quad \quad \models \quad \quad
h\circ \Pushmap{u} \,\, : \,\, {\smallstar_A} \arrow[rr] && S'
\end{tikzcd}
$$
The fact that the morphism~(\ref{equation/image-structure}) is opcartesian
simply means that the operation is reversible, and thus induces a bijection
$$
\Ecategory_{v\circ u:A\to C} (\smallstar_A,S') \quad \cong \quad \Ecategory_{v:B\to C} (\push{u}{A},S')
$$
for every morphism $v:B\to B'$ and every object $S'$ in the fiber of $B'$.
%


\begin{proposition}\label{proposition/from-image-structure-to-push-structure}
For every morphism $u:A\to B$ of the category~$\Bcategory$,
every functor $p:\Ecategory\to\Bcategory$ with an image structure $(\smallstar,\Pushmap{})$
comes equipped with a family of morphisms
\begin{equation}\label{equation/post}
\begin{tikzcd}[column sep=.8em]
v \,\, : \,\, B \arrow[rr] && B'
\quad \quad \models \quad \quad
\postp{v} \,\, : \,\, \push{u}{A} \arrow[rr] &&  \push{v \circ u}{A}
\end{tikzcd}
\end{equation}
indexed by the morphisms $v:B\to B'$ of the category $\Bcategory$,
and a family of morphisms
\begin{equation}\label{equation/pre}
\begin{tikzcd}[column sep=.8em]
\id{B} \,\, : \,\, B \arrow[rr] && B
\quad \quad \models \quad \quad
\prep{w} \,\, : \,\, \push{u\circ w}{A'} \arrow[rr] && \push{u}{A}
\end{tikzcd}
\end{equation}
indexed by the morphisms $w:A'\to A$ of the category $\Bcategory$.
These morphisms make a series of diagrams commute.
First of all, the three coherence diagrams below commute
\begin{equation}\label{equation/two-coherence-diagrams}
\begin{array}{ccc}
	\begin{tikzcd}[column sep=.8em,row sep=.3em]
	&&
	\push{u}{A}
	\ar[dddd,"{\postp{v}}"]
	\\
	\\
	\sectionp{A}\ar[rruu, bend left, "{\Pushmap{u}}"]
	\ar[rrdd,bend right, "{\Pushmap{vu}}",swap]
	\ar[rr,phantom, "{(a)}"{description}]
	& & \,
	\\
	\\
	&&
	\push{v\circ u}{A}
	\end{tikzcd}
&
\begin{tikzcd}[column sep=1.2em,row sep=1.4em]
	\sectionp{A'}
	\ar[dd,"{\sectionp{w}}",swap]
	\ar[rr,"{\Pushmap{u\circ w}}"]
	\ar[rrdd,phantom, "{(b)}"{description}]
	&&
	\push{u\circ w}{A'}
	\ar[dd,"{\prep{w}}"]
	\\
	\\
	\sectionp{A}
	\ar[rr,"{\Pushmap{u}}"{swap}]
	&&
	\push{u}{A}
	\end{tikzcd}
&	
	\begin{tikzcd}[column sep = 2em, row sep = 3em]
\push{u\circ w}{A'} \ar[d,"{\prep{w}}", swap] \ar[r,"{\postp{v}}"] 
\ar[rd,phantom, "{(c)}"{description}]
& \push{v\compose u\compose w}{A'} \ar[d,"{\prep{w}}"] \\
\push{u}{A}\ar[r,"{\postp{v}}"{swap}] & \push{v\circ u}{A}
\end{tikzcd} 
\end{array}
\end{equation}
for every path
$
\begin{tikzcd}[column sep = .5em]
A'\arrow[rr,"{w}"] && A\arrow[rr,"{u}"] && B\arrow[rr,"{v}"] && B'
\end{tikzcd}
$
in the category $\Bcategory$.
Then, the functorial nature of (\ref{equation/post})
and (\ref{equation/pre}) is ensured by the fact that the diagrams below commute
\begin{equation}\label{equation/compositionality-of-left-right-actions}
	\begin{array}{ccc}
	\begin{tikzcd}[column sep = .3em, row sep = 2em]
	& \push{v\compose u}{A} \ar[rd,"{\postp{v'}}"] \\
	\push{u}{A} \ar[rr,"{\postp{(v' \compose v)}}"] \ar[ru,"{\postp{v}}"] & & \push{v' \compose v \compose u}{A}
	\end{tikzcd}
	& \quad &	
	\begin{tikzcd}[column sep = .3em, row sep = 2em]
	& \push{u\compose w}{A'} \ar[rd,"{\prep{w}}"] \\
	\push{u\compose w \compose w'}{A''} \ar[rr,"{\prep{(w \compose w')}}"] \ar[ru,"{\prep{w'}}"]  & & \push{u}{A}
	\end{tikzcd} 
	\end{array}
\end{equation}
for every pair of paths
$
\begin{tikzcd}[column sep = .6em]
A\arrow[rr,"{u}"] && B\arrow[rr,"{v}"] && B'\arrow[rr,"{v'}"] && B''
\end{tikzcd}
$
and
$
\begin{tikzcd}[column sep = .6em]
A''\arrow[rr,"{w'}"] && A'\arrow[rr,"{w}"] && A\arrow[rr,"{u}"] && B
\end{tikzcd}
$
of the category~$\Bcategory$, and moreover,
that the morphisms
\begin{equation}\label{equation/identity-of-left-right-actions}
\begin{array}{c}
\begin{tikzcd}[column sep=.8em]
\id{B} \,\, : \,\, B \arrow[rr] && B
\quad \quad \models \quad \quad
\postp{(id_B)} \,\, : \,\, \push{u}{A} \arrow[rr] && \push{u}{A}
\end{tikzcd}
\\
\begin{tikzcd}[column sep=.8em]
\id{B} \,\, : \,\, B \arrow[rr] && B
\quad \quad \models \quad \quad
\prep{(id_A)} \,\, : \,\, \push{u}{A} \arrow[rr] && \push{u}{A}
\end{tikzcd}
\end{array}
\end{equation}
coincide with the identity, for every morphism $u:A\to B$ of the basis category $\Bcategory$.
\end{proposition}
\begin{proof}
The two morphisms (\ref{equation/post}) and (\ref{equation/pre})
are defined by the universal property of the family of opcartesian morphisms $\Pushmap{}$ 
defining the image structure, as the unique morphisms $\postp{v}$ and $\prep{w}$
making the two diagrams commute in~(\ref{equation/two-coherence-diagrams}-$ab$).
The three coherence properties~(\ref{equation/two-coherence-diagrams}-$c$)
(\ref{equation/compositionality-of-left-right-actions})
and~(\ref{equation/identity-of-left-right-actions}) follow easily from the definition 
of the two morphisms~(\ref{equation/post}) and~(\ref{equation/pre}).
\end{proof}
We deduce from the statement (Prop.~\ref{proposition/from-image-structure-to-push-structure})
just established that 
\begin{corollary}\label{corollary/image-functor}
Every functor $p:\Ecategory\to\Bcategory$ with an image structure comes with a functor
$$
\begin{tikzcd}[column sep =1em]
\imagefunctor \,\, : \,\, \Barrowcategory \arrow[rr] && \Ecategory
\end{tikzcd}
$$
called the \textbf{image functor} associated to the image structure.
\end{corollary}
\begin{proof}
The image functor transports every object $u:A\to B$ of the arrow category $\Barrowcategory$
to the object $\push{u}{A}$ defined by the image structure, 
and every morphism 
$$\begin{tikzcd}[column sep =1em, row sep = 1em]
(v,w) \quad : \quad (A,B,u:A\to B)\arrow[rr] && (A',B',u':A'\to B')
\end{tikzcd}
$$
to the composite morphism below in the category $\Ecategory$
$$
\begin{tikzcd}[column sep=1em]
v \,\, : \,\, B \arrow[rr] && B'
\quad  \models \quad
{\push{u}{A}} \arrow[rr,"{\postp{v}}"] && 
{\push{v\compose u}{A}}
\arrow[rr,"{\id{}}"]
&&
{\push{u'\compose w}{A}} \arrow[rr,"{\prep{w}}"] 
&&
{\push{u'}{A'}}
\end{tikzcd}
$$
The functoriality of $\imagefunctor$ follows from the coherence properties
(\ref{equation/two-coherence-diagrams}-$c$)
(\ref{equation/compositionality-of-left-right-actions})
and (\ref{equation/identity-of-left-right-actions}) 
established in Prop.~\ref{proposition/from-image-structure-to-push-structure}.
\end{proof}
The resulting image functor $\imagefunctor:\Barrowcategory\to\Ecategory$
extends the section $\smallstar:\Bcategory\to\Ecategory$,
in the expected sense that the diagram below commutes:
\begin{equation}\label{equation/main-diagram}
\begin{tikzcd}[column sep=1.1em,row sep=.8em]
&& \Bcategory
\\
\\
&& \Ecategory\arrow[uu,"{p}"{swap}]
\\
\\
\Bcategory\arrow[rrrr,"{\arrowid}"] \arrow[rruu,"{\smallstar}"] \arrow[rruuuu,"{\id{\Bcategory}}",bend left=60] 
&&&& 
\Barrowcategory\arrow[lluu,"{\imagefunctor}"{swap}]\arrow[lluuuu,"{\codfunctor}"{swap},bend right=60]
\end{tikzcd}
\end{equation}
In particular, as explained in the introduction, the diagram~(\ref{equation/Barrowcategory-image-functor}) commutes by definition of the image functor.
%
%
%
Note that every Grothendieck opfibration $p:\Ecategory\to\Bcategory$ with a section $\smallstar:\Bcategory\to\Ecategory$ comes equipped with an image structure, which is canonical when the opfibration is cloven.

\section{Comprehension structures with image}\label{section/comprehension-with-image}
In order to work in full generality, and to include the case of the 2-categories of algebras and coalgebras 
treated in \S\ref{section/application-to-induction-and-coinduction}, we find convenient
to generalize our definition Def.~\ref{definition/imagestructureonCat} of image structure 
for a functor $p:\Ecategory\to\Bcategory$ in the specific case $\Ktwocategory=\Cat$ 
to any morphism $p:\Eobject\to\Bobject$ in a 2-category $\Ktwocategory$.

\begin{definition}[Image structure]\label{definition/image-structure-in-K}
An image structure on a morphism $p:\Eobject\to\Bobject$ in a 2-category $\Ktwocategory$ is a section $\smallstar:\Bobject\to\Eobject$
equipped with a family~$\lambda$ of 2-cells
\begin{equation}\label{equation/image-structure-in-K}
\begin{tikzcd}[column sep=.8em]
{\lambda_u} \,\,\,\, : \,\,\,\, 
{\smallstar_a} \arrow[rr,double,-implies] && \push{u}{a}
\,\,\,\, : \,\,\,\,
\Xobject \arrow[rr] && \Eobject
\end{tikzcd}
\end{equation}
indexed by the objects $\Xobject$, the morphisms $a,b:\Xobject\to \Bobject$ and the 2-cells $u:a\Rightarrow b$ of the 2-category $\Ktwocategory$, where the morphism $\smallstar_a$ is defined as the composite $\smallstar\circ a:\Xobject\to \Eobject$.
One requires moreover that each 2-cell $\lambda_u$
defines an opcartesian morphism
\begin{equation}\label{equation/image-structure-in-K-next}
\begin{tikzcd}[column sep=.8em]
u \,\, : \,\, a \arrow[rr,double,-implies] && b
\quad \models \quad
{\Pushmap{u}} \,\, : \,\, {\smallstar_a} \arrow[rr,double,-implies] && \push{u}{a}
\end{tikzcd}
\end{equation}
with respect to the postcomposition functor 
$$\begin{tikzcd}[column sep=.8em]
\Ktwocategory(\Xobject,p)
\quad : \quad
\Ktwocategory(\Xobject,\Eobject)
\arrow[rr] && 
\Ktwocategory(\Xobject,\Bobject)
\end{tikzcd}
$$
above the morphism $u:a\Rightarrow b$ in the category $\Ktwocategory(\Xobject,\Bobject)$.
%
We also ask for convenience that $\lambda_{u}:{\smallstar_a}\Rightarrow\push{u}{a}$
coincides with the identity 2-cell when $u:a\Rightarrow a$ is the identity 2-cell.
\end{definition}
Note that every morphism $p:\Eobject\to\Bobject$ 
with an image structure $(\smallstar,\lambda)$
to an object $\Bobject$ equipped 
with a path-object $(\Bobject^{\to},\beta)$ in the 2-category $\Ktwocategory$
comes equipped with a morphism $\imagefunctor:\Bobject^{\to}\to \Eobject$ defined as 
$\imagefunctor=\push{\arrowcell}{\arrowdom}$, and thus satisfying the equality:
\begin{equation}\label{equation/image-in-K}
\begin{array}{ccc}
\begin{tikzcd}[column sep =2em, row sep = 1em]
&&
\Bobject
\arrow[dd,"{\smallstar}"]
\\
\\
{\Bobject^{\to}}
\arrow[rruu,yshift=.2em,"{\arrowdom}",bend left=30,""{swap,name=source}] 
\arrow[rrdd,yshift=-.2em,"{\arrowcod}"{swap},bend right=30,""{name=target}] 
&& 
\Eobject
\arrow[dd,"{p}"]
\\
\\
&& 
\Bobject
\arrow[from=source,to=target,double,-implies,shorten <>=2pt,"{\arrowcell}"{description,pos=.5}]
\end{tikzcd}
& \quad\quad = \quad\quad &
\begin{tikzcd}[column sep =2em, row sep = 1em]
&&
\Bobject
\arrow[dd,"{\smallstar}"]
\\
\\
{\Bobject^{\to}}
\arrow[rruu,yshift=.2em,"{\arrowdom}",bend left=30,""{swap,name=uppersource}] 
\arrow[rr,"{\imagefunctor}"{swap},""{name=uppertarget},""{swap,name=downsource}] 
\arrow[rrdd,yshift=-.2em,"{\arrowcod}"{swap},bend right=30,""{name=downtarget}] 
&& 
\Eobject
\arrow[dd,"{p}"]
\\
\\
&& 
\Bobject
\arrow[from=uppersource,to=uppertarget,double,-implies,shorten <>=2pt,"{\lambda_{\arrowcell}}"{pos=.5}]
\end{tikzcd}
\end{array}
\end{equation}
%
%
Moreover, the resulting image morphism makes 
the counterpart of diagram~(\ref{equation/main-diagram})
commute for the same reason as in the specific case
of the 2-category $\Ktwocategory=\Cat$.
For that reason, the morphism $\imagefunctor$ may be seen
as a morphism 
\begin{equation}\label{equation/image-morphism-in-slice}
\begin{tikzcd}[column sep=.5em]
\imagefunctor
\quad : \quad
(\Bobject^{\to},\arrowcod)
\arrow[rr]
&&
(\Eobject,p)
\end{tikzcd}
\end{equation}
in the slice 2-category $\Ktwocategory/\Bobject$,
defined as the expected sub-2-category of $\Ktwocategory/\Ktwocategory$
whose objects are the morphisms $p:\Eobject\to\Bobject$ with codomain $\Bobject$.
%
We are now in the position of defining the notion of \emph{comprehension structure with image} 
at that 2-categorical level of generality.

\begin{definition}\label{definition/comprehension-structure-with-image-in-K}
Suppose given a morphism $p:\Eobject\to\Bobject$ with an image structure
on an object $\Bobject$ equipped with a path-object $(\Bobject^{\to},\beta)$.
A comprehension structure with image on the morphism $p:\Eobject\to\Bobject$
is a right adjoint $\jacobscomprehension:(\Eobject,p)\to (\Bobject^{\to},\arrowcod)$
to the morphism $\imagefunctor:(\Bobject^{\to},\arrowcod)\to(\Eobject,p)$
defined in~(\ref{equation/image-morphism-in-slice})
in the slice 2-category $\Ktwocategory/B$.
\end{definition}
\noindent
A comprehension structure with image on $p : \Eobject \to \Bobject$ 
in the sense of Def.~\ref{definition/comprehension-structure-with-image-in-K}
comes equipped with a pair of adjunctions $\arrowid:\Bobject\leftrightarrows \Bobject^{\to}:\arrowdom$ and $\imagefunctor:\Bobject^{\to}\leftrightarrows \Eobject:\jacobscomprehension$
in the 2-category~$\Ktwocategory$.
From that, one easily deduces that
%
%
\begin{proposition}
Every comprehension structure with image 
(Def.~\ref{definition/comprehension-structure-with-image-in-K})
induces a comprehension structure with section
defined as $\smallstar=\imagefunctor\circ\arrowid:\Bobject\to\Eobject$
(Def.~\ref{definition/comprehension-structure-with-section-in-K}),
where the right adjoint functor $\comprehensionzero:\Eobject\to\Bobject$ is defined 
as the composite $\comprehensionzero=\arrowdom\circ\jacobscomprehension$.
\end{proposition}
%
%
We establish now the converse property which extends to every 2-category $\Ktwocategory$
the property stated in the introduction (Prop. \ref{proposition/comprehension-structure-with-image-is-imagreand-comprehension-structure-with-section}) in the specific case of $\Ktwocategory=\Cat$.
The statement extends~\cite{fumex-phdthesis}, lemma 2.2.10
by relaxing the assumption that $p:\Eobject\to\Bobject$ is a bifibration.
\begin{proposition}\label{proposition/comprehension-structure-with-section-with-image-structure-in-K} 
Suppose that $p:\Eobject\to\Bobject$ has an image structure in the 2-category $\Ktwocategory$
(in the sense of Def.~\ref{definition/image-structure-in-K}).
In that case, every comprehension structure with section 
(in the sense of Def. \ref{definition/comprehension-structure-with-section-in-K})
defines a comprehension structure with image 
(in the sense of Def.~\ref{definition/comprehension-structure-with-image-in-K}).
\end{proposition}
\noindent
\begin{proof}
Suppose that $p:\Eobject\to\Bobject$ has an image structure 
and at the same time a comprehension structure with section $\smallstar:\Eobject\to\Bobject$
in the 2-category $\Ktwocategory$.
The 2-cell $\iota:\comprehensionzero\to p:\Eobject\to\Bobject$ mentioned in Prop.~\ref{proposition/lifting}
defines a morphism in the 2-category $\Ktwocategory(\Eobject,\Bobject)$,
and thus an object in the 2-category $\Ktwocategory(\Eobject,\Bobject)^{\to}$.
By definition of the path object $(\Bobject^{\to},\beta)$ in Def.~\ref{definition/path-object},
the 2-cell $\iota:\comprehensionzero\to p:\Eobject\to\Bobject$ induces
an object of the 2-category $\Ktwocategory(\Eobject,\Bobject^{\to})$,
and thus a morphism noted $\jacobscomprehension:\Eobject\to\Bobject^{\to}$
and characterized by the equation 
$$
\begin{tikzcd}[column sep=.8em]
\arrowcell\circ \jacobscomprehension
 \,\,\,\, = \,\,\,\, \iota 
 \,\,\,\, : \,\,\,\,  
{\comprehensionzero} \arrow[rr,double,-implies] && p
\,\,\,\, : \,\,\,\,
\Eobject \arrow[rr] && \Bobject
\end{tikzcd}
$$
%
%
We want to show that this morphism $\jacobscomprehension:\Eobject\to\Bobject^{\to}$
is right adjoint to the morphism $\imagefunctor:\Bobject^{\to}\to\Eobject$
in the 2-category~$\Ktwocategory$.
%
To that purpose, we consider an object $\Xobject$ of the 2-category~$\Ktwocategory$
and a pair of morphisms $u:\Xobject\to \Bobject^{\to}$ and $S:\Xobject\to\Eobject$,
and we exhibit a one-to-one relationship (see \cite{Mellies09panorama}, Section 5.11)
between the 2-cells~$\varphi$ and~$\psi$ of the form:
\begin{equation}\label{equation/pair-of-two-cells}
\begin{array}{ccc}
\begin{tikzcd}[column sep=1em, row sep=1em]
& {} &
\Bobjectarrow
\arrow{dd}[]{\imagefunctor} 
 \\
\Xobject \arrow[urr,bend left,"{u}",""{swap,name=source}]
\arrow[drr,bend right,"{S}"{swap},""{name=target}]
\\
& {} & 
\Eobject
\arrow[from=source,to=target,double,-implies,shorten <>=1pt,"{\varphi}"{pos=.5}]
\end{tikzcd}
& \quad \quad \quad  & 
\begin{tikzcd}[column sep=1em, row sep=1em]
& {} &
\Bobjectarrow
 \\
\Xobject \arrow[urr,bend left,"{u}",""{swap,name=source}]
\arrow[drr,bend right,"{S}"{swap},""{name=target}]
\\
& {} & 
\Eobject
\arrow[uu,"{\jacobscomprehension}"{swap}]
\arrow[from=source,to=target,double,-implies,shorten <>=1pt,"{\psi}"{pos=.5}]
\end{tikzcd}
\end{array}
\end{equation}
The key observation is that a 2-cell $\psi$ of that form is the same thing as 
a 2-cell $(\psi_1,\psi_2)$ in the 2-category $\Ktwocategory//\Ktwocategory$
between the composite morphisms:
$$
\begin{array}{ccc}
\begin{tikzcd}[column sep=2em, row sep=1em]
{\Xobject}
\arrow[rr,"{u}"]
\arrow[dd,"{\id{\Xobject}}"{swap}]
&&
\Bobjectarrow
\arrow[rr,"{\arrowdom}"]
\arrow[dd,"{\arrowcod}"{swap},""{name=target}]
&&
{\Bobject}\arrow[dd,"{\id{\Bobject}}",""{name=source,swap}]
\\
\\
{\Xobject}
\arrow[rr,"{\arrowcod\circ u}"{swap}]
&&
{\Bobject}
\arrow[rr,"{\id{\Bobject}}"{swap}]
&&
{\Bobject}
\arrow[from=source,to=target,double,-implies,shorten <>=1pt,"{\arrowcell}"{swap,pos=.5}]
\end{tikzcd}
& 
\begin{tikzcd}[column sep = 2em]
\,\arrow[rr,double,-implies,"{(\psi_1,\psi_2)}"] && \,
\end{tikzcd}
&
\begin{tikzcd}[column sep=2em, row sep=1em]
{\Xobject}
\arrow[rr,"{S}"]
\arrow[dd,"{\id{\Xobject}}"{swap}]
&&
\Eobject
\arrow[rr,"{\comprehensionzero}"]
\arrow[dd,"{p}"{swap},""{name=target}]
&&
{\Bobject}\arrow[dd,"{\id{\Bobject}}",""{name=source,swap}]
\\
\\
{\Xobject}
\arrow[rr,"{p\circ S}"{swap}]
&&
{\Bobject}
\arrow[rr,"{\id{\Bobject}}"{swap}]
&&
{\Bobject}
\arrow[from=source,to=target,double,-implies,shorten <>=1pt,"{\iota}"{swap,pos=.5}]
\end{tikzcd}
\end{array}
$$
It follows from the existence of the adjunction in $\Ktwocategory//\Ktwocategory$
established in Prop.~\ref{proposition/comprehension-structure-with-section-in-K}
that there is a one-to-one relationship between the pairs of 2-cells $(\psi_1,\psi_2)$ 
in the 2-category $\Ktwocategory$ of the form above, and the pairs $(\varphi_1,\varphi_2)$
of 2-cells in the 2-category $\Ktwocategory$
defining a 2-cell $(\varphi_1,\varphi_2)$ in the 2-category $\Ktwocategory//\Ktwocategory$
between the composite morphisms:
$$
\begin{array}{ccc}
\begin{tikzcd}[column sep=2em, row sep=1em]
{\Xobject}
\arrow[rr,"{u}"]
\arrow[dd,"{\id{\Xobject}}"{swap}]
&&
\Bobjectarrow
\arrow[rr,"{\arrowdom}"]
\arrow[dd,"{\arrowcod}"{swap},""{name=target}]
&&
{\Bobject}\arrow[dd,"{\id{\Bobject}}",""{name=source,swap}]
\arrow[rr,"{\smallstar}"]
&&
{\Eobject}\arrow[dd,"{p}"]
\\
\\
{\Xobject}
\arrow[rr,"{\arrowcod\circ u}"{swap}]
&&
{\Bobject}
\arrow[rr,"{\id{\Bobject}}"{swap}]
&&
{\Bobject}
\arrow[rr,"{\id{\Bobject}}"{swap}]
&&
{\Bobject}
\arrow[from=source,to=target,double,-implies,shorten <>=1pt,"{\arrowcell}"{swap,pos=.5}]
\end{tikzcd}
& 
\begin{tikzcd}[column sep = 2em]
\,\arrow[rr,double,-implies,"{(\varphi_1,\varphi_2)}"] && \,
\end{tikzcd}
&
\begin{tikzcd}[column sep=2em, row sep=1em]
{\Xobject}
\arrow[rr,"{S}"]
\arrow[dd,"{\id{\Xobject}}"{swap}]
&&
\Eobject
\arrow[dd,"{p}",""{name=target}]
\\
\\
{\Xobject}
\arrow[rr,"{p\circ S}"{swap}]
&&
{\Bobject}
\end{tikzcd}
\end{array}
$$
The definition of the morphism $\imagefunctor:\Bobjectarrow\to\Ecategory$ and
the cartesianity of the 2-cell $\lambda_{\arrowcell}$ in~(\ref{equation/image-in-K}) 
with respect to the functor $\Ktwocategory(\Xobject,p):\Ktwocategory(\Xobject,\Eobject)\to\Ktwocategory(\Xobject,\Bobject)$
implies that there is a one-to-one relationship between the pairs of 2-cells $(\varphi_1,\varphi_2)$ 
in the 2-category $\Ktwocategory//\Ktwocategory$ above, and the 2-cells $\varphi$ of the form (\ref{equation/pair-of-two-cells})
in the 2-category $\Ktwocategory$.
The end of the proof is easy.
\end{proof}
\section{Illustration: inductive reasoning on functor algebras and dually, coinductive reasoning on functor coalgebras}\label{section/application-to-induction-and-coinduction}\label{section/application}
Suppose given a functor $p:\Ecategory\to\Bcategory$ between two categories $\Bcategory$ and $\Ecategory$
equipped with endofunctors $\Ffunctor:\Bcategory\to\Bcategory$ and $\Gfunctor:\Ecategory\to\Ecategory$ 
and a distributivity law of the form~(\ref{equation/distributivity-law}).
A well-known result by Beck \cite{beck1969distributive} states that the distributivity law 
$\delta: \Ffunctor\circ p\Rightarrow p\circ \Gfunctor$ describes one specific lifting of the functor 
$p:\Ecategory\to\Bcategory$ to a functor $p':\Alg{\Gfunctor}(\Ecategory)\to\Alg{\Ffunctor}(\Bcategory)$ 
between the underlying categories of algebras, in such a way that the diagram below commutes:
$$
\begin{tikzcd}[column sep=2em,row sep=.8em]
{\Alg{\Gfunctor}(\Ecategory)}  \ar[dd, "{U}"{swap}] \ar[rr,"{p'}"] & {} &  {\Alg{\Ffunctor}(\Bcategory)}  \ar[dd,"{U}"] 
\\
\\
{\Ecategory}  \ar[rr,"p"] & {} &  {\Bcategory}
\end{tikzcd}
$$
where $U$ denotes in both cases the forgetful functor.
One main reason for working at a 2-categorical level as we do in the present paper 
is to provide us with a simple and elegant recipe to characterize in just the same spirit 
inherited from Beck \cite{beck1969distributive} when a comprehension structure
with section $\smallstar:\Bcategory\to\Ecategory$ on the functor $p:\Ecategory\to\Bcategory$ 
lifts to a comprehension structure with section $\smallstar:\Alg{\Ffunctor}(\Bcategory)\to\Alg{\Gfunctor}(\Ecategory)$ 
on the functor $p:\Alg{\Gfunctor}(\Ecategory)\to\Alg{\Ffunctor}(\Bcategory)$.
To that purpose, we consider the 2-category $\Endo{\Cat}$ with objects the categories equipped with endofunctors,
and with morphisms the functors equipped with a distributivity law \`a la Beck.
Note that $\Endo{\Cat}$ may be defined as the full sub-2-category of $\ArrowLax{\Cat}$ 
whose objects are of the form $(\Ccategory,\Ccategory,\Gfunctor:\Ccategory\to\Ccategory)$.
%
%
%
This leads us the question of characterizing when a comprehension structure 
with section $\smallstar:\Bcategory\to\Ecategory$ on the functor $p:\Ecategory\to\Bcategory$
lifts to a \emph{comprehension structure with section}
in the 2-category $\Endo{\Cat}$, where this definition should be understood
in the 2-categorical sense of \S\ref{section/comprehension-with-section},
Def.~\ref{definition/comprehension-structure-with-section-in-K}.
We establish that

\begin{proposition}
Suppose given a comprehension structure with section $(\comprehensionzero,\smallstar)$
in $\Cat$ on a functor $p:\Ecategory\to\Bcategory$ between categories $\Bcategory$ and $\Ecategory$
equipped with endofunctors $S:\Bcategory\to\Bcategory$ and $T:\Ecategory\to\Ecategory$.
There is a one-to-one correspondence between the liftings to $\Endo{\Cat}$
of the comprehension structure with section $(\comprehensionzero,\smallstar)$
and the pairs of distributivity laws
$$\begin{array}{ccc}
\begin{tikzcd}[column sep = .8em, row sep = .8em]
\delta \,\, : \,\, \Ffunctor\circ p \arrow[rr,double,-implies] && p\circ \Gfunctor
\,\, : \,\, \Ecategory\arrow[rr] && \Bcategory
\end{tikzcd}
& \quad\quad\quad &
\begin{tikzcd}[column sep = .8em, row sep = .8em]
\sigma \,\, : \,\, \Gfunctor\circ \smallstar \arrow[rr,double,-implies] && \smallstar\circ \Ffunctor
\,\, : \,\, \Bcategory\arrow[rr] && \Ecategory
\end{tikzcd}
\end{array}
$$
such that (1) the composite natural transformation
$$
\begin{tikzcd}[column sep = 1.4em, row sep = .8em]
\Ffunctor \arrow[rr,double,-implies,"{equal}"] && 
{\Ffunctor\circ p\circ\smallstar} \arrow[rr,double,-implies,"{\delta\circ\smallstar}"] &&
{p\circ \Gfunctor\circ\smallstar} \arrow[rr,double,-implies,"{p\circ\sigma}"] &&
{p\circ \smallstar\circ \Ffunctor} \arrow[rr,double,-implies,"{equal}"] && {\Ffunctor}
\end{tikzcd}
$$
is the identity
and (2) the natural transformation $\sigma$ is reversible.
\end{proposition}
It is worth mentioning that, in this situation, the comprehension functor
$\comprehensionzero:\Ecategory\to\Bcategory$ lifts as the pair
$(\comprehensionzero,\widetilde{\sigma}):(\Ecategory,\Gfunctor)\to(\Bcategory,\Ffunctor)$
where the distributivity law $\widetilde{\sigma}:\Ffunctor\circ\comprehensionzero\Rightarrow\comprehensionzero\circ \Gfunctor$
is defined as the mate of the inverse $\sigma^{-1}: \smallstar\circ \Ffunctor\Rightarrow \Gfunctor\circ \smallstar$.
We obtain that in this situation
\begin{corollary}\label{corollary/soundness-end-of-paper}
The comprehension structure with section $\smallstar:\Bcategory\to\Ecategory$
lifts to a comprehension structure with section $(\smallstar,\sigma):(\Bcategory,\Ffunctor)\to(\Ecategory,\Gfunctor)$
which is left adjoint to comprehension $(\comprehensionzero,\widetilde{\sigma}):(\Ecategory,\Gfunctor)\to(\Bcategory,\Ffunctor)$
and thus transports the initial $\Ffunctor$-algebra $\mu\Ffunctor$ to the initial $\Gfunctor$-algebra $\mu\Gfunctor=\smallstar_{\mu\Ffunctor}$.
\end{corollary}
The result extends the main soundness theorem established by Fumex, Ghani 
and Johann in~\cite{ghani-johann-fumex}, Thm 2.10.
It establishes in their terminology (see \cite{fumex-phdthesis}, Def. 4.3.1)
that $G$ defines an induction scheme for $\mu F$ in $p$.
The approach translates immediately by duality to the case of quotient structures,
and provides in just the same way necessary and sufficient conditions to be in the situation 
of Corollary~\ref{corollary/quotient-structures} and to characterize
the terminal $\Ffunctor$-coalgebra as $\nu\Ffunctor=\smallstar_{\nu\Ffunctor}$.

\section{Conclusion}\label{section/conclusion}
Our main purpose and achievement in this paper is to exhibit the 2-categorical structures
secretly at work in the 1-categorical approach to comprehension structures traditionally
found in categorical logic.
Our work was motivated by the fibered approach to induction on algebras 
and coinduction on coalgebras recently developed by Fumex, Ghani and Johann~\cite{ghani-johann-fumex,fumex-phdthesis}.
We understand our 2-categorical approach and statements (Cor.~\ref{corollary/comprehension-structures}, \ref{corollary/quotient-structures} 
and \ref{corollary/soundness-end-of-paper})
as providing the clean conceptual foundations underlying their soundness theorems.
For lack of space, we did not treat here the proof-theoretical aspects of our 2-categorical description
of comprehension structures.
A natural direction would be to start from the recent multicategorical approach 
to induction~\cite{jafarrahmani-master-thesis} developed in the fibered style
of Melli\`es and Zeilberger's refinement systems~\cite{MelliesZeilberger15}.
We leave that for future work.

\section*{Acknowledgments}
The research underlying this article was partially supported by the ERC Advanced Grant DuaLL, number 670624.
%


\bibliography{comprehension-biblio}

\appendix


\section{Four alternative notions of comprehension structures}
For the sake of completeness, we give the list below 
of four well-recognized notions of comprehension structures
appearing in the literature.

\paragraph*{Jacobs comprehension categories.} 
The notion of comprehension category was introduced by Jacobs 
(\cite{jacobs-paper}, Def 4.1, p. 181, and \cite{jacobs:book}, chapter 10.4, page 613).
A comprehension category is defined there as a functor 
$\jacobscomprehension:\Ecategory\to\Bcategory^{\to}$ satisfying that
\begin{enumerate}
\item the functor $p:\Ecategory\to\Bcategory$ defined as the composite functor $p=\codfunctor\circ \jacobscomprehension$ is a Grothendieck fibration,
\item the functor $\jacobscomprehension:\Ecategory\to\Bcategory^{\to}$ 
is cartesian in the sense that it transports every $p$-cartesian morphism
of $\Ecategory$ to a $\codfunctor$-cartesian morphism of $\Bcategory^{\to}$,
which may be equivalently defined as a \textit{pullback square} in the category $\Bcategory$.
\end{enumerate}
A comprehension category is thus the same thing as a comprehension structure 
in the sense of Def.~\ref{definition/comprehension-structure}
where the underlying functor $p:\Ecategory\to\Bcategory$ is a Grothendieck fibration,
and where the functor $\jacobscomprehension:\Ecategory\to\Bcategory^{\to}$
induced from the functor $\comprehensionzero:\Ecategory\to\Bcategory$
and from the natural transformation $\injectionzero:\comprehensionzero\Rightarrow p$
transports every $p$-cartesian morphism of $\Ecategory$
to a pullback square in the category $\Bcategory$.
Note that in that case, the equality $p=\codfunctor\circ \jacobscomprehension$ holds by construction.

%

\paragraph*{Ehrhard $D$-categories.} 
The notion of $D$-category was introduced by Ehrhard \cite{ehrhard}.
Ehrhard's $D$-categories are also called \textit{comprehension category with units}
by Jacobs \cite{jacobs-paper} def. 4.12, and \emph{Ehrhard comprehension category}
by Moss \cite{moss2018dialectica}, p 22.
A pre-$D$-category is defined in \cite{ehrhard} (Section 2.1, def. 5) 
as a functor $p:\Ecategory\to \Bcategory$ equipped with a right adjoint
functor $\smallstar:\Bcategory\to \Ecategory$ such that the counit of the adjunction
$p\dashv \smallstar$
an isomorphism, or equivalently, that the functor $\smallstar$ is fully faithful.
The functor $\smallstar:\Bcategory\to \Ecategory$ may be thus seen as a section
of the functor $p:\Ecategory\to \Bcategory$ up to isomorphism.
In the definition of a pre-$D$-category, the functor $\smallstar:\Bcategory\to \Ecategory$
should also come equipped with a right adjoint functor $\comprehensionzero: \Ecategory\to\Bcategory$.
Finally, a $D$-category is defined in \cite{ehrhard} (Section 2.1, def. 5)
as a pre-$D$-category where the functor $p:\Ecategory\to \Bcategory$
is a Grothendieck fibration.

In his later reformulation \cite{jacobs-paper} of the notion of $D$-category, 
Jacobs makes the extra assumption that the counit of the adjunction $p\dashv \smallstar$ 
is the identity, and not just an isomorphism.
This implies in particular that the functor $\smallstar:\Bcategory\to\Ecategory$ 
is a section of the functor $p:\Ecategory\to\Bcategory$.
A $D$-category is thus defined in \cite{jacobs-paper} Def. 4.12 as a Grothendieck fibration 
$p:\Ecategory\to \Bcategory$ equipped with a terminal object functor 
$\smallstar:\Bcategory\to \Ecategory$
which has a right adjoint noted $\comprehensionzero: \Ecategory\to\Bcategory$.
Here, by terminal object functor $s:\Bcategory\to \Ecategory$,
one means a section of the functor $p:\Ecategory\to\Bcategory$
which transports every object $A$ of the basis category $\Bcategory$
to a terminal object of the fiber $\Ecategory_{A}$ of the object~$A$
with respect to the functor $p:\Ecategory\to\Bcategory$.
Note in particular that the terminal object functor $\smallstar:\Bcategory\to \Ecategory$
is fully faithful and right adjoint to the functor $p:\Ecategory\to\Bcategory$.

A $D$-category in that sense is thus the same thing as a comprehension structure
with section (Def. \ref{definition/comprehension-structure-with-section})
where the functor $p:\Ecategory\to\Bcategory$ is a Grothendieck fibration,
and where the section $\smallstar:\Bcategory\to \Ecategory$ is moreover right adjoint 
to $p:\Ecategory\to\Bcategory$.
As mentioned above, this last point means that the section $\smallstar:\Bcategory\to \Ecategory$
is the terminal object function which associates to every object $A$ of the category $\Bcategory$
the terminal object in its fiber $\Ecategory_{A}$ with respect to the functor $p:\Ecategory\to\Bcategory$.

\paragraph*{Fumex $tC$-opfibrations.}
The notion of $tC$-opfibration was introduced by Fumex in his PhD thesis, (\cite{fumex-phdthesis} p. 38, def. 2.2.2.)
A $tC$-category is defined there as a Grothendieck opfibration $p:\Ecategory\to\Bcategory$
with a fully faithful section $\sectionp{}:\Bcategory\to\Ecategory$.
The section $\sectionp{}:\Bcategory\to\Ecategory$ is moreover required to have a right adjoint
noted $\comprehensionzero: \Ecategory\to\Bcategory$.
Note that, given an object $A$ of the basis category $\Bcategory$,
one does not require that the object $s_A$ is terminal in the fiber $\Ecategory_A$ of the object~$A$.
This is one main difference with Ehrhard's notion of $D$-category.

A $tC$-category is thus the same thing as a comprehension structure with section
in the sense of Def.~\ref{definition/comprehension-structure-with-section}
where the functor $p:\Ecategory\to\Bcategory$ is a Grothendieck opfibration
and where the section $\sectionp{}:\Bcategory\to\Ecategory$ is moreover fully faithful.
At this stage, it is important to observe that every Grothendieck opfibration has an image structure
in the sense of Def.~\ref{definition/imagestructureonCat}, or equivalently, 
in the sense of Def.~\ref{definition/image-structure-in-K} for the specific case $\Ktwocategory=\Cat$.
From this follows, by Prop.~\ref{proposition/comprehension-structure-with-section-with-image-structure-in-K},
that a $tC$-category is in fact the same thing as a comprehension structure 
with image in sense of Def.~\ref{definition/comprehension-structure-with-image},
where the functor $p:\Ecategory\to\Bcategory$ is moreover a Grothendieck opfibration
and where the section $\sectionp{}:\Bcategory\to\Ecategory$ is fully faithful.



\paragraph*{Lawvere categories.} 
The notion of Lawvere category was introduced by Jacobs in \cite{jacobs-paper}, p 190,
as a way to reflect the work by Lawvere \cite{lawvere-hyperdoctrines} on hyperdoctrines in categorical logic.
A Lawvere category is defined as a Grothendieck bifibration $p:\Ecategory\to\Bcategory$
with a terminal object in each fiber, defining a functor $\sectionp{}:\Bcategory\to\Ecategory$, 
and such that the (ordinary) functor
$$
f \mapsto \Sigma_f \sectionp{(dom f)} : \Bcategory^{\to}\to \Ecategory
$$
induced by the left fibration structure has a right adjoint $\comprehensionzero:\Ecategory\to\Bcategory^{\to}$, 
verifying $\codfunctor\circ \comprehensionzero=p$,
and such that the unit and counit are vertical (their image by $\codfunctor$ and $p$ is the identity).
Note that every Lawvere category is a $tC$-opfibration in the sense of Fumex \cite{fumex-phdthesis}.
A Lawvere category is thus the same thing as a comprehension structure 
with image in sense of Def.~\ref{definition/comprehension-structure-with-image}
where the functor $p:\Ecategory\to\Bcategory$ is a Grothendieck bifibration
and where the section $\sectionp{}:\Bcategory\to\Ecategory$ is the terminal object functor.
\end{document}